\documentclass[11pt,reqno]{amsart}
\usepackage{amsmath,amssymb,verbatim,geometry,color}
\usepackage[pdftex,hyperindex]{hyperref}
\usepackage[pdftex]{graphicx}
\newcommand{\A}{\mathbf{A}}

\newcommand{\C}{\mathbf{C}}

\newcommand{\Q}{\mathbf{Q}}
\newcommand{\R}{\mathbf{R}}
\newcommand{\Z}{\mathbf{Z}}
\newcommand{\N}{\mathbf{N}}
\renewcommand{\P}{\mathbf{P}}

\newcommand{\fa}{\mathfrak{a}}
\newcommand{\fb}{\mathfrak{b}}
\newcommand{\fm}{\mathfrak{m}}

\newcommand{\cB}{\mathcal{B}}

\newcommand{\cF}{\mathcal{F}}
\newcommand{\cG}{\mathcal{G}}

\newcommand{\cJ}{\mathcal{J}}

\newcommand{\cO}{\mathcal{O}}

\newcommand{\cR}{\mathcal{R}}

\newcommand{\cV}{\mathcal{V}}

\newcommand{\cW}{\mathcal{W}}

\renewcommand{\a}{\alpha}
\renewcommand{\b}{\beta}
\renewcommand{\d}{\delta}
\newcommand{\e}{\varepsilon}

\newcommand{\g}{\gamma}
\newcommand{\la}{\lambda}

\newcommand{\p}{\psi}

\newcommand{\eg}{{\rm e.g.\ }} 
\newcommand{\ie}{{\rm i.e.\ }}

\DeclareMathOperator{\Aut}{Aut}

\DeclareMathOperator{\codim}{codim}

\DeclareMathOperator{\Hom}{Hom}
\DeclareMathOperator{\lct}{lct}

\DeclareMathOperator{\Spec}{Spec}
\DeclareMathOperator{\supp}{supp}
\DeclareMathOperator{\vol}{vol}

\DeclareMathOperator{\Int}{Int}

\DeclareMathOperator{\im}{Im}

\DeclareMathOperator{\ord}{ord}

\DeclareMathOperator{\NS}{NS}

\DeclareMathOperator{\Proj}{Proj}

\DeclareMathOperator{\Ric}{Ric}

\DeclareMathOperator{\Span}{span}

\DeclareMathOperator{\Val}{Val}

\DeclareMathOperator{\ver}{Vert}
\DeclareMathOperator{\QM}{QM}

\DeclareMathOperator{\init}{in}
\DeclareMathOperator{\Jac}{Jac}

\newcommand{\ab}{\mathfrak{a}_\bullet}
\newcommand{\bs}[1]{\fb\big( | #1 | \big)}

\renewcommand{\div}{\mathrm{div}}

\newcommand{\D}{\Delta}

\numberwithin{equation}{section}       

\newtheorem{prop} {Proposition} [section]
\newtheorem{thm}[prop] {Theorem} 
 
\newtheorem{rem}[prop] {Remark} 

\newtheorem{lem}[prop] {Lemma}
\newtheorem{cor}[prop]{Corollary}
\newtheorem{prop-def}[prop]{Proposition-Definition}

\newtheorem{conj}[prop]{Conjecture}
 
\newtheorem*{thmA}{Theorem A} 
 
\newtheorem*{thmB}{Theorem B} 
\newtheorem*{thmC}{Theorem C} 
\newtheorem*{thmD}{Theorem D} 
\newtheorem*{thmE}{Theorem E} 
\newtheorem*{thmF}{Theorem F}

\newtheorem{rmk}[prop]{Remark}
\theoremstyle{rem}

\newtheorem*{ackn}{Acknowledgment}

\begin{document}

\newpage 
 \pagenumbering{arabic}

\title{Thresholds, valuations, and K-stability}
\date{\today}

\author{Harold Blum \and Mattias Jonsson}

\address{Department of Mathematics\\
  University of Michigan\\
  Ann Arbor, MI 48109--1043\\
  USA}
\email{blum@umich.edu}
\email{mattiasj@umich.edu}


\begin{abstract}
  Let $X$ be a normal complex projective variety with at worst klt singularities, and 
  $L$ a big line bundle on $X$. We use valuations to study the log canonical threshold
  of $L$, as well as another invariant, the stability threshold. The latter generalizes a 
  notion by Fujita and Odaka, and can be used to characterize when a $\Q$-Fano variety is 
  $K$-semistable or uniformly K-stable. 
  It can also be used to generalize volume bounds due to Fujita and Liu.
  The two thresholds can be written as infima of certain functionals on 
  the space of valuations on $X$. When $L$ is ample, we prove that these infima are attained.
  In the toric case, toric valuations achieve these infima, and we 
  obtain simple expressions for the two thresholds in terms of the moment polytope of $L$.
\end{abstract}

\maketitle

\setcounter{tocdepth}{1}
\tableofcontents
%
%
%
%
\section*{Introduction}
Let $X$ be a normal complex projective variety of dimension $n$ with at worst klt singularities, 
and let $L$ a big line bundle on $X$. We shall consider two natural ``thresholds'' of $L$, both
involving the asymptotics of the singularities of the linear system $|mL|$ as $m\to\infty$.

First, the \emph{log canonical threshold} of $L$, measuring the \emph{worst}
singularities, is defined by 
\begin{equation*}
  \a(L)=\inf\{\lct(D)\mid \text{$D$ effective $\Q$-divisor, $D\sim_\Q L$}\}.
\end{equation*}
where $\lct(D)$ is the log canonical threshold of $D$; see~\eg~\cite{CS08}. 
It is an algebraic version of the $\a$-invariant defined
analytically by Tian~\cite{Tian87} when $X$ is Fano and $L=-K_X$.

The second invariant measures the ``average'' singularities
and was introduced by Fujita and Odaka
in the Fano case, where it is relevant for K-stability, see~\cite{FO16,PW16}.
Following~\cite{FO16} 
we say that an effective $\Q$-divisor $D\sim_\Q L$ on 
$X$ is of \emph{$m$-basis type}, where $m\ge 1$, 
if there exists a basis $s_1,\dots,s_{N_m}$ of $H^0(X,mL)$ such that 
\begin{equation*}
  D=\frac{\{s_1=0\}+\{s_2=0\}+\dots+\{s_{N_m}=0\}}{mN_m},
\end{equation*}
where $N_m=h^0(X,mL)$.
Define 
\begin{equation*}
  \d_m(L)=\inf\{\lct(D)\mid D\sim_\Q L\ \text{of $m$-basis type}\}.
\end{equation*}
Our first main result is
\begin{thmA}
  For any big line bundle $L$, the limit $\d(L)=\lim_{m\to\infty}\d_m(L)$ exists, and 
  \begin{equation*}
    \a(L)
    \le\d(L)
    \le(n+1)\a(L).
  \end{equation*}
  Further, the numbers $\a(L)$ and $\d(L)$ are strictly positive and only depend on
  the numerical equivalence class of $L$. 
  When $L$ is ample, the stronger inequality $\d(L)\ge\frac{n+1}{n}\a(L)$ holds.
\end{thmA}

We call $\d(L)$ the \emph{stability threshold}\footnote{The idea of the stability threshold $\d(L)$, with a slightly different definition, was suggested to the second author by R.~Berman~\cite{BermComm}.} of $L$ (in the literature it is now also commonly referred to as the $\delta$\emph{-invariant}). It can also be defined for $\Q$-line bundles $L$ by 
$\d(L):=r\d(rL)$ for any $r\ge1$ such that $rL$ is a line bundle; see Remark~\ref{R104}.

The following result, which verifies Conjecture~0.4 and strengthens Theorem~0.3 
of~\cite{FO16}, relates the stability threshold to the $K$-stability of a $\Q$-Fano variety:
\begin{thmB}
  Let $X$ be a $\Q$-Fano variety.
  \begin{itemize}
  \item[(i)]
    $X$ is K-semistable iff $\d(-K_X)\ge 1$;
  \item[(ii)]
    $X$ is uniformly K-stable iff $\d(-K_X)>1$. 
  \end{itemize}
\end{thmB}
More precisely, the reverse implications are due to Fujita and Odaka~\cite{FO16};  what is new are the direct implications.

The notion of uniform K-stability was introduced in~\cite{BHJ1,Der16}. 
As a special case of the Yau--Tian--Donaldson conjecture, it was proved in~\cite{BBJ15}
that a Fano manifold $X$ without nontrivial vector fields is uniformly $K$-stable
iff $X$ admits a K\"ahler-Einstein metric. 
The latter equivalence was extended to (possibly) singular $\Q$-Fano varieties without nontrivial vector field in~\cite{LTW19}, and general singular $\Q$-Fano varieties in~\cite{Li19}. The result in~\cite{Li19} says that 
a $\Q$-Fano variety admits a K\"ahler--Einstein metric iff $X$ is uniformly
K-polystable. For smooth $X$, this result was proved earlier (using different methods, and with uniform K-polystability replaced by K-polystability) in~\cite{CDS15,Tian15}.

For a general ample line bundle $L$ on a smooth complex projective variety, the 
stability threshold $\d(L)$ detects \emph{Ding stability} in the sense of~\cite{nakstab} and has the following analytic interpretation.\footnote{However, $\d(L)$ is not expected to be directly
related to the $K$-stability of the pair $(X,L)$.}
Let $\b(L)$ be the greatest Ricci lower bound, \ie the supremum
of all $\b>0$ such that there exists a K\"ahler form $\omega\in c_1(L)$ with
$\Ric\omega\ge\beta\omega$, see~\cite{Tian92,Rub08,Rub09,Sze11}.
Then $\b(L)=\min\{\d(L),s(L)\}$, where $s(L)=\sup\{s\in\R\mid -K_X-sL\ \text{nef}\}$
is the nef threshold if $L$, see~\cite[Theorem~D]{BBJ18} and also~\cite[Appendix]{CRZ18}.

Theorems~A and~B imply that if $X$ is a $\Q$-Fano variety and 
$\a(-K_X)\ge\frac{n}{n+1}$ (resp.\ $>\frac{n}{n+1})$, then $X$ is K-semistable 
(resp.\ uniformly K-stable), thus recovering results in~\cite{OSa,BHJ1,Der16,FO16},
that can be viewed as algebraic versions of Tian's theorem in~\cite{Tian97}.
See also~\cite{Fujitaalpha} for the case $\a(-K_X)=\frac{n}{n+1}$,
and~\cite{Der15} for more general polarizations.

\medskip
Our approach to the two thresholds $\a(L)$ and $\d(L)$ is through \emph{valuations}.
Let $\Val_X$ be the set of (real) valuations on the function field on $X$ that are trivial 
on the ground field $\C$, and equip $\Val_X$ with the topology of pointwise convergence. 
To any $v\in\Val_X$ we can associate several invariants. 

First, we have the \emph{log discrepancy} $A(v)=A_X(v)$.
Here we only describe it when $v$ is divisorial; see~\cite{BdFFU} for the general case.
Let $E$ be a prime divisor over $X$, \ie $E\subset Y$ is a prime divisor, 
where $Y$ is a normal variety with a proper birational morphism $\pi\colon Y\to X$.
In this case, the log discrepancy of the divisorial valuation $\ord_E$ is given by 
$A(\ord_E)=1+\ord_E(K_{Y/X})$, where $K_{Y/X}$ is the relative canonical divisor. 

Second, following~\cite{BKMS}, we have asymptotic invariants of valuations 
that depend on a big line bundle $L$. For simplicity assume $H^0(X,L)\ne0$.
To any $v\in\Val_X$ and any nonzero section
$s\in H^0(X,L)$ we can associate a positive real number 
$v(s)\in\R_+$. This induces a decreasing real \emph{filtration} 
$\cF_v$ on $H^0(X,L)$, given by 
\begin{equation*}
  \cF^t_vH^0(X,L)=\{s\in H^0(X,L)\mid v(s)\ge t\}
\end{equation*}
for $t\ge0$. Define the \emph{vanishing sequence} or \emph{sequence of jumping numbers}
\begin{equation*}
  0=a_1(L,v)\le a_2(L,v)\le\dots\le a_N(L,v)=a_{\max}(L,v)
\end{equation*}
of (the filtration associated to) $v$ on $L$ by 
\begin{equation*}
  a_j(L,v)=\inf\{t\in\R\mid \codim\cF^t_vH^0(X,L)\ge j\}.
\end{equation*}
Thus the set of jumping numbers equals the set of all values $v(s)$, 
$s\in H^0(X,L)\setminus\{0\}$.

For $m\ge 1$, consider the rescaled maximum and average jumping numbers of $v$ on $mL$:
\begin{equation*}
  T_m(v)=\frac1m a_{\max}(mL,v)
  \quad\text{and}\quad
  S_m(v)=\frac1{mN_m}\sum_{j=1}^{N_m}a_j(mL,v),
\end{equation*}
where $N_m=h^0(X,mL)$.
Using Okounkov bodies one shows that the limits
\begin{equation*}
  S(v)=\lim_{m\to\infty}S_m(v)
  \quad\text{and}\quad
  T(v)=\lim_{m\to\infty}T_m(v)
\end{equation*}
exist. The resulting functions $S,T\colon\Val_X\to\R_+\cup\{+\infty\}$ are lower
semicontinuous. They are finite on the locus $A(v)<\infty$.
For a divisorial valuation $v=\ord_E$ as above, 
the invariant $T(\ord_E)$ can be viewed as a pseudoeffective threshold:
\begin{equation*}
  T(\ord_E)=\sup\{t>0\mid \pi^\ast L-tE\ \text {is pseudoeffective}\}
\end{equation*}
whereas $S(\ord_E)$ is an ``integrated volume''.
\begin{equation*}
  S(\ord_E)=\vol(L)^{-1}\int_0^\infty\vol(\pi^\ast L-tE)\,dt.
\end{equation*}
The invariants $S(\ord_E)$ and $T(\ord_E)$ play an important role in the work of
K.~Fujita~\cite{Fujitavalcrit}, C.~Li~\cite{Liequivariant}, and Y.~Liu~\cite{Liu16}, 
see Remark~\ref{R102}.

\medskip
The next result shows that log canonical and stability thresholds can be 
computed using the invariants of valuations above:
\begin{thmC}\label{t:thmc}
  For any big line bundle $L$ on $X$, we have
  \begin{equation*}
    \a(L)=\inf_v\frac{A(v)}{T(v)}
    =\inf_E\frac{A(\ord_E)}{T(\ord_E)}
    \quad\text{and}\quad
    \d(L)=\inf_v\frac{A(v)}{S(v)}
    =\inf_E\frac{A(\ord_E)}{S(\ord_E)},
  \end{equation*}
  where $v$ ranges over nontrivial valuations with $A(v)<\infty$,
  and $E$ over prime divisors over~$X$.
\end{thmC}
While the formulas for $\a(L)$ follow quite easily from the definitions
(see also~\cite[\S3.2]{Amb16}),
the ones for $\d(L)$ (as well as the fact that the limit $\d(L)=\lim_m\d_m(L)$ exists) 
are more subtle and use the concavity of the function on the
Okounkov body of $L$ defined by the filtration associated to the valuation $v$ as
in~\cite{BC11,BKMS}; see also~\cite{WN12}. 

\smallskip
Theorem~B follows from the second formula for $\d(L)$ above and results in~\cite{Fujitavalcrit} and~\cite{Liequivariant}.

As for Theorem~A, the estimates between $\a(L)$ and $\d(L)$ in Theorem~A
follow from estimates $\frac1{n+1}T(v)\le S(v)\le T(v)$ that are proved along the way.
When $L$ is ample and $v$ is divisorial, the stronger inequality $S(v)\le\frac{n}{n+1}T(v)$
was proved by Fujita~\cite{Fujitaplt}. 
We deduce from results in~\cite{BKMS} that the invariants $S(v)$ and $T(v)$ only depend
on the numerical equivalence class of $L$. By Theorem~C, the same is therefore true for
the thresholds $\a(L)$ and $\d(L)$. 
The proof that $\a(L)>0$ can be reduced to the case when $L$ is ample,
where it is known~\cite{Tian87,BHJ1}.
By the estimates in Theorem~A, it follows that $\d(L)>0$.

\medskip
We can also bound the volume of a line bundle in terms of the stability 
threshold:
\begin{thmD}
  Let $L$ be a big line bundle. Then we have 
  \[
    \vol(L) \le \left(\frac{n+1}{n}\right)^n \d(L)^{-n} \widehat{\vol}(v).
  \]
  for any valuation $v$ on $X$ centered at a closed point.
\end{thmD}
Here $\widehat{\vol}(v)$ is the \emph{normalized volume} of $v$,
introduced by C.~Li~\cite{LiMinimizing}.
When $X$ is a $\Q$-Fano variety and $L=-K_X$, Theorem~D
generalizes the volume bounds found in~\cite{FujitaOptimal}  and~\cite{Liu16},
in which $X$ is assumed $K$-semistable, so that $\d(L)\ge1$.
These volume bounds were explored in~\cite{SS17} and~\cite{LX17}.

\medskip
Next we investigate whether the infima in Theorem~C are attained.
We say that a valuation $v\in\Val_X$  
\emph{computes} the log canonical threshold if $\frac{A(v)}{T(v)}=\a(L)$.
Similarly, $v$ computes the stability threshold if $\frac{A(v)}{S(v)}=\d(L)$.
\begin{thmE}
  If $L$ is ample, then there exist valuations  with finite log discrepancy computing 
  the log-canonical threshold and the stability threshold,  respectively.
\end{thmE}
This theorem can be viewed as a global analogue of the main result in~\cite{Blu16b},
where the existence of a valuation minimizing the normalized volume is established.
It is also reminiscent of results in~\cite{jonmus} on the existence of valuations computing
log canonical thresholds of graded sequence of ideals, and related to a
recent result  by Birkar~\cite{Bir16} on the existence of $\Q$-divisors achieving the 
infimum in the definition of $\lct(L)$ in the $\Q$-Fano case (see also~\cite{ACS18}), and to the existence of
optimal destabilizing test configurations~\cite{Don02,Sze08,Oda15,DS16}.

\medskip
Unlike the case in~\cite{jonmus}, Theorem~E does not seem to directly 
follow from an argument
involving compactness and semicontinuity. Instead we use a ``generic limit'' construction
as in~\cite{Blu16b}. For example, given a sequence of $(v_i)_i$ of valuations on $X$ 
such that $\lim_iA(v_i)/S(v_i)=\d(L)$, we want to find a valuation $v^*$ with 
$A(v^*)/S(v^*)=\d(L)$. Roughly speaking, we do this by first extracting a limit filtration
$\cF^*$ on the section ring of $L$ from the filtrations $\cF_{v_i}$; then $v^*$ is chosen,
using~\cite{jonmus}, so as to compute the log canonical threshold of the graded sequence
of base ideals associated to $\cF^*$. 
To make all of this work, we need uniform versions of the Fujita approximation results 
from~\cite{BC11}; these are proved using multiplier ideals.

\medskip
As a global analogue to conjectures in~\cite{jonmus} we conjecture that 
any valuation computing one of the thresholds $\a(L)$ or $\d(L)$ must be quasimonomial.
While this conjecture seems difficult in general, we establish it
when $X$ is a surface with at worst canonical singularities, see Proposition~\ref{p:twodim}.
Using results in~\cite{Blu16a,Fujitaplt}, we prove 
in Proposition~\ref{p:pltblowup} that any \emph{divisorial} 
valuation computing $\a(L)$ or $\d(L)$ is associated to a log canonical type divisor over $X$. 
When $L$ is ample, any divisorial valuation computing $\d(L)$ is in fact 
associated to a plt type divisor over $X$.

\medskip
Finally we treat the case when $X$ is a toric variety, associated to a complete fan $\D$,
and $L$ is ample. 
We can embed $N_\R\subset\Val_X$ as the set of toric (or monomial) valuations.
The primitive lattice points $v_i$, $1\le i\le d$, of the 1-dimensional cones of $\D$
then correspond to the divisorial valuations $\ord_{D_i}$, where $D_i$ are the corresponding
torus invariant divisors.

Let $P\subset M_\R$ be the polytope associated to $L$. To each $u\in P\cap M_\Q$ 
is associated an effective torus invariant $\Q$-divisor $D_u\sim_\Q L$ on $X$.
 \begin{thmF}
   The log-canonical and stability thresholds of $L$ are given by 
  \begin{equation*}
    \a(L) = \min_{u\in \ver(P)}\lct(D_u)
    \quad\text{and}\quad
    \d(L) = \lct(D_{\bar{u}}),
  \end{equation*}
  where $\bar{u}\in M_\Q$ denotes the barycenter of $P$, and $\ver(P)\subset M_\Q$ 
  the set of vertices of $P$. 
  Furthermore, $\alpha(L)$ (resp. $\delta(L)$) is computed by one of the valuations 
  $v_1, \dots, v_d$. 
\end{thmF}
The main difficulty in the proof is to show that the two thresholds are computed by 
toric valuations. For $\a(L)$, this is not so hard, and the formula in the theorem 
is in fact already known; see~\cite{Son05,LSY15} and also~\cite{CS08,Del15,Amb16}.
In the case of $\d(L)$, we use initial degenerations, a global adaptation 
of methods utilized in~\cite{Musgraded,Blu16b}.

When $X$ is a toric $\Q$-Fano variety and $L=-K_X$, Theorem~F implies that 
$X$ is $K$-semistable iff the barycenter of $P$ is the origin. 
This result was previously proven by analytic methods in~\cite{BB13,Berm16}
 and also  follows from~\cite[Theorem 1.4]{LX16}, 
which was proven algebraically. 

Additionally, we give a formula for $\d(-K_X)$ in terms of the polytope $P$. When $X$ is a smooth toric Fano variety, $\d(-K_X)$ agrees with the formula in~\cite{Li11}
 for the greatest Ricci lower bound (see~\cite{Tian92,Sze11}).

\medskip

We expect the results in this paper admit equivariant versions, relative to a subgroup 
$G\subset\Aut(X,L)$.
It should also be possible to bound the 
stability threshold $\d(L)$ from below in terms of a ``Berman-Gibbs'' invariant, 
as in~\cite{FO16}; see also~\cite{Berm13,FujitaBG}.

\medskip

Since the first version of this paper, there have been many developments related to the  topics in this paper. 
\begin{itemize}
\item The stability threshold has played an important role in a number of papers. 
For instance, see~\cite{BL18,BX18,CPS18,CRZ18,CZ18,CP18,Gol19}.
\item It was recently shown in~\cite{Xu19} that a weak version of
\cite[Conjecture B]{jonmus} holds. 
This result implies that any valuation computing $\delta(L)$ is quasimonomial; 
see Remark~\ref{deltaconj}.
\item In the thesis of the first author, 
the results in this paper were extended to the setting of klt pairs $(X,B)$~\cite{BluThesis} (see also~\cite{CP18}). 
The arguments from this paper go through to the more general setting with little to no substantive changes. 
\end{itemize}

\medskip
The paper is organized as follows. After some general background 
in~\S\ref{S112}, we study filtrations in~\S\ref{S103} and global invariants of valuations 
in~\S\ref{S105}, mainly following~\cite{BC11,BKMS}. 
We are then ready to prove the first main results on thresholds, Theorems~A-D, 
in~\S\ref{S106}. The uniform Fujita approximation results appear in~\S\ref{s:fujitaapprox}
and Theorem~E is proved in~\S\ref{s:theoremE} using the generic limit construction.
Finally, the toric case is analyzed in~\S\ref{S113}.
%
%
%
%
\begin{ackn}
  We thank R.~Berman, K.~Fujita, C.~Li and Y.~Odaka for comments on a preliminary 
  version of the paper. The first author wishes to thank Y.~Liu for fruitful discussions, 
  and his advisor, M.~Musta\c{t}\u{a}, for teaching him many of the tools that went into 
  this project. The second author has benefitted from countless discussions with 
  R.~Berman and S.~Boucksom.
  This research was supported by NSF grants DMS-0943832 and DMS-1600011,
  and by BSF grant 2014268.
\end{ackn}
%
%
%
%
%
%
\section{Background}\label{S112}
%
%
%
%
\subsection{Conventions} 
We work over $\C$. A \emph{variety} is 
an irreducible, reduced, separated scheme of finite type.
An \emph{ideal} on a variety $X$ is a coherent ideal sheaf $\fa\subset\cO_X$.
We frequently use additive notation for line bundles, \eg~$mL:=L^{\otimes m}$.

We use the convention ${\N= \{0, 1, 2,\dots \}}$, $\N^*=\N\setminus\{0\}$,
$\R_+=[0,+\infty)$, $\R_+^*=\R_+\setminus\{0\}$.
In an inclusion $A\subset B$ between sets, the case of equality is allowed.

%
%
%
%
\subsection{Valuations}\label{S107}
Let $X$ be a normal projective variety. A \emph{valuation on $X$} will mean 
a valuation $v\colon\C(X)^* \to \R$ that is trivial on $\C$.
By projectivity, $v$ admits a unique \emph{center} on $X$,  that is, 
a point $ \xi:=c_X(v)\in X$ such that $v\ge0$ on $\cO_{X,\xi}$
and $v>0$ on the maximal ideal of $\cO_{X,\xi}$. 
We use the convention that $v(0) = \infty$.

Following~\cite{jonmus,BdFFU} we define $\Val_X$ as the set of valuations on $X$ 
and equip it with the topology of pointwise convergence.\footnote{This is the weakest topology such that  for each $f\in \C(X)^*$ the evaluation map $\varphi_f: \Val_X \to \R$ defined by $\varphi_f(v) := v(f)$  is continuous. See~\cite[Section 4.1]{jonmus} for further details. }
We define a partial ordering on $\Val_X$ by 
$v\le w$ iff $c_X(w)\in\overline{c_X(v)}$ and $v(f)\le w(f)$ for $f\in\cO_{X,c_X(w)}$. 
The unique minimal element is the \emph{trivial} valuation on $X$.
We write $\Val_X^*$ for the set of nontrivial valuations on $X$.

If $Y\to X$ is a proper birational morphism, with $Y$ normal, and $E\subset Y$
is a prime divisor (called a \emph{prime divisor over $X$}), then $E$ defines a
valuation $\ord_E\colon \C(X)^*\to\Z$ in $\Val_X$ given by order of vanishing at the generic point of $E$. Any valuation of the form 
$v=c\ord_E$ with $c\in\R_{>0}$ will be called \emph{divisorial}.

To any valuation $v\in \Val_X$ and $\la\in \R_+$ there is an associated \emph{valuation ideal}
defined by  $\fa_\la(v) := \{f \in \cO_X \, \vert \, v(f) \ge\la \}$. 
If $v$ is divisorial, then Izumi's inequality (see~\cite{HS01}) shows that there exists $c>0$
such that $\fa_\la(v)\subset\fm_\xi^{\lceil c\la\rceil}$ for any $\la\in\R_+$,  where $\xi=c_X(v)$.

For an ideal $\fa\subset \cO_X$ and $v\in \Val_X$,  we set 
\[
v(\fa) := \min\{v(f)\, \vert\, f\in \fa \cdot \cO_{X, c_X(v)} \} \in [0, +\infty].\]
We can also make sense of $v(s)$ when $L$ is a line bundle and $s\in H^0(X, L)$. After trivializing $L$ at $c_X(v)$, we write $v(s)$ for the value of the local function corresponding to $s$ under this trivialization;
this is independent of the choice of trivialization.

We similarly define $v(D)$ where $D$ is an effective $\Q$-Cartier divisor on $X$.  Pick $m\ge 1$ such that $mD$ is Cartier and set $v(D)=m^{-1}v(f)$, where $f$
is a local equation of $mD$ at the center of $v$ on $X$. 
 Equivalently, 
$v(D)=m^{-1}v(s)$, where $s$ is the canonical section of $\cO_X(mD)$ defining $mD$.
%
%
%
%
\subsection{Graded sequences of ideals}
A \emph{graded sequence of ideals} is a sequence $\fa_\bullet=(\fa_p)_{p \in \N^*}$ 
of ideals on $X$ satisfying $\fa_p\cdot \fa_q\subset \fa_{p+q}$ for all $p,q\in \N^*$. We will always assume $\fa_p\ne (0)$ for some $p\in \N^*$.  
We write $M(\ab) := \{p \in \N^*\mid\fa_p \neq (0)\}$.  
By convention,  $\fa_{0}:=\cO_X$. 

Given a valuation $v\in \Val_X$, it follows from Fekete's Lemma that the limit 
\begin{equation*}
  v(\fa_\bullet) := \lim_{M(\ab)\ni p\to\infty} \frac{v(\fa_p)}{p}
\end{equation*}
exists, and equals $\inf_{p\in M(\ab)} v(\fa_p) /p$; see~\cite{jonmus}. 

A graded sequence $\fa_\bullet$ of ideals will be called \emph{nontrivial} if there 
exists a divisorial valuation $v$ such that $v(\fa_\bullet)>0$. 
By Izumi's inequality, this is equivalent to the 
existence of a point $\xi\in X$ and $c>0$ such that $\fa_p\subset\fm_\xi^{\lceil cp\rceil}$ 
for all $p\in\N$.

If $v$ is a nontrivial valuation on $X$, then $\ab(v) := \{\fa_p(v)\}_{p \in\N^*}$ 
is a graded sequence of ideals. In this case, $v(\ab(v))=1$~\cite[Lemma 3.5]{Blu16b}.
%
%
%
%
\subsection{Volume}
Let $v$ be a valuation centered at a closed point $\xi\in X$.
The \emph{volume} of $v$ is
\begin{equation*}
  \vol(v):=\lim_{\la\to+\infty}\frac{\ell(\cO_{X,\xi}/\fa_\la(v))}{\la^n/n!}\in[0,+\infty),
\end{equation*}
the existence of the limit being a consequence of~\cite{Cut13}.
The volume function is homogenous of order $-n$, 
\ie~$\vol(tv)=t^{-n}\vol(v)$ for $t>0$.
%
%
%
%
\subsection{Log discrepancy}
Let $X$ be a normal variety such that the canonical divisor $K_X$ is $\Q$-Cartier. 
If $\pi\colon Y\to X$ is a projective birational morphism with $Y$ normal, and $E\subset Y$ 
a prime divisor, then the \emph{log discrepancy} of $\ord_E$ is defined by 
$A_X(\ord_E):=1+\ord_E(K_{Y/X})$, where $K_{Y/X}:=K_Y-\pi^*K_X$ is the relative canonical divisor. We say $X$ has \emph{klt} singularities  if $A_X(\ord_E)>0$ for all prime divisors $E$ over $X$.

Now assume $X$ has klt singularities. As explained in~\cite{BdFFU} 
(building upon~\cite{hiro,jonmus}), the log discrepancy can be naturally extended to a 
lower semicontinuous function $A=A_X\colon\Val_X\to[0,+\infty]$
that is homogeneous of order 1, \ie~$A(t v)=t A(v)$ for $\la\in \R_{+}$.

We have $A(v)=0$ iff $v$ is the trivial valuation.
The log-discrepancy $A_X$ depends on $X$, but if $Y\to X$ is as above, then
$A_X(v)=A_{Y}(v)+v(K_{Y/X})$; hence $A_{Y}(v)<\infty$ iff $A_X(v)<\infty$.

If $A(v)<\infty$, then $\fa_\bullet(v)$ is a nontrivial graded sequence of ideals
by the Izumi-Skoda inequality, see~\cite[Proposition~2.3]{LiMinimizing}.

\subsection{Fano varieties and K-stability}
A variety $X$ is called \emph{$\Q$-Fano} if $X$ is projective with klt singularities
and $-K_X$ is ample. See ~\cite{BHJ1} for the definition of K-semistability and uniform K-stability of a $\Q$-Fano variety in terms of invariants associated to  test configurations. 
In this paper, we will use a characterization of these notions in terms of invariants of divisorial valuations~\cite{Liequivariant,Fujitavalcrit} (see Section 4.3).
%
%
%
\subsection{Normalized volume}
In~\cite{LiMinimizing}, C.~Li introduced the \emph{normalized volume} of a valuation $v$
centered at a closed point on $X$ as
$\widehat{\vol}(v):=A(v)^n\vol(v)$ when $A(v)<\infty$, and 
$\widehat{\vol}(v):=\infty$ when $A(v)=\infty$. This is a homogeneous function of degree 0 
on $\Val_X$. The first author proved in~\cite{Blu16b} that for any closed point $\xi\in X$,
the normalized volume function restricted to valuations centered at $\xi$
attains its infimum.
%
%
%
%
\subsection{Log canonical thresholds}
Let $X$ be a klt variety. Given a nonzero ideal $\fa \subset \cO_X$, the \emph{log canonical threshold} of $\fa$ is given by 
\begin{equation*}
\lct(\fa):= \inf_{v} \frac{A(v)}{v(\fa)} = \inf_{E} \frac{A(\ord_E)}{\ord_E(\fa)}
\
\end{equation*}
where the first infimum runs through all $v\in \Val^*_X$ and the second through all prime 
divisors $E$ over $X$.  In fact, it suffices to consider $E$ on a fixed log resolution of~$\fa$.   

In the above infima we use the convention that if $v(\fa) = 0$, then $A(v)/v(\fa) = +\infty$. Thus, $\lct(\cO_X) = +\infty$. By convention, we set $\lct((0)) = 0$. 
 
We say a valuation $v^\ast\in \Val_X^*$ \emph{computes} 
$\lct(\fa)$ if $\lct(\fa) =A(v^\ast)/v^\ast(\fa)$. 
There always exists a divisor $E$ over $X$ such that $\ord_E$ computes $\lct(\fa)$. 
 
Given a graded sequence of ideals $\ab$ on $X$, we set 
\begin{equation*}
  \lct(\ab):= \lim_{M(\ab) \ni m \to \infty}  m \cdot \lct(\fa_m) = \sup_{m \ge 1} m \cdot \lct(\fa_m).
\end{equation*}
By~\cite{jonmus}, we have 
\begin{equation*}
  \lct(\ab)= \inf_{v\in\Val_X^*} \frac{A(v)}{v(\ab)},
\end{equation*}
We say $v^\ast \in \Val_X$ \emph{computes} $\lct(\ab)$ if $\lct(\ab)=A(v^\ast)/v^\ast(\ab)$. 
Such valuations always exist: 
see~\cite[Theorem A]{jonmus} for the smooth case and~\cite[Theorem B.1]{Blu16b} 
for the klt case.

We now state two elementary lemmas that will be used in future sections.
\begin{lem}\label{l:lctval}
  If $v$ is a nontrivial valuation on $X$, then $\lct(\ab(v)) \le A(v) $ and equality holds iff $v$ computes $\lct(\ab(v))$. 
\end{lem}
\begin{proof}
  The statement is an immediate consequence of the definition of $\lct(\ab(v))$ and the fact that $v(\ab(v))=1$. 
 \end{proof}
 
\begin{lem}\label{l:validealincl}
Let $v\in \Val_X$ and $\ab$ a graded sequence of ideals on $X$. If $v(\ab)\ge 1$, then $\fa_p \subset \fa_p(v)$ for all $p \in \N$. 
\end{lem}
\begin{proof}
Since $1 \le v(\ab) = \inf_p v(\fa_p)/p$, we see that $p \le v(\fa_p)$. Therefore, $\fa_p \subset \fa_p(v)$.
\end{proof}
%
%
%
%
%
%
\section{Linear series, filtrations, and Okounkov bodies}\label{S103}
In this section we recall facts about linear series, filtrations, and 
Okounkov bodies, following~\cite{LM09,KK12,BC11,Bou14}. 
The new results are Lemma~\ref{L102} and Corollary~\ref{C101}.

Let $X$ be a normal projective variety of dimension $n$ and $L$ a big 
line bundle on $X$. Set 
\begin{equation*}
  R_m:=H^0(X,mL)
  \quad\text{and}\quad
  N_m:=\dim_\C R_m
\end{equation*}
for $m\in\N$, and write 
$M(L)\subset\N$ for the semigroup of $m\in\N$ for which $N_m>0$.
Since $L$ is big, we have $m\in M(L)$ for $m\gg1$.
Write 
\begin{equation*}
  R=R(X,L)=\bigoplus_mR_m=\bigoplus_mH^0(X,mL)
\end{equation*}
for the section ring of $L$.
%
%
%
%
\subsection{Graded linear series}
A \emph{graded linear series} of $L$ is a graded $\C$-subalgebra
\begin{equation*}
  V_\bullet
  =\bigoplus_m V_m
  \subset\bigoplus_m R_m
  = R.
\end{equation*}

We say $V_\bullet$ \emph{contains an ample series} if $V_m\ne0$ for $m\gg0$, 
and there exists a decomposition $L=A+E$ with $A$ an ample $\Q$-line bundle
and $E$ an effective $\Q$-divisor such that 
\begin{equation*}
  H^0(X,mA)\subset V_m\subset H^0(X,mL)=R_m
\end{equation*}
for all sufficiently divisible $m$.
%
%
%
%
\subsection{Okounkov bodies}\label{S102}
Fix a system $z=(z_1,\dots,z_n)$ of parameters centered at a regular closed point $\xi$
of $X$. This defines a real rank-$n$ valuation 
\begin{equation*}
  \ord_z\colon\cO_{X,\xi}\setminus\{0\}\to\N^n,
\end{equation*}
where $\N^n$ is equipped with the lexicographic ordering.
As in~\S\ref{S107} we also define $\ord_z(s)$ for any nonzero section $s\in R_m$.

Now consider a nonzero graded linear series $V_\bullet\subset R(X,L)$.
For $m\in\N$, the subset  
\begin{equation*}
  \Gamma_m
  :=\Gamma_m(V_\bullet)
  :=\ord_z(V_m\setminus\{0\})
  \subset\N^n
\end{equation*}
has cardinality $\dim_\C V_m$, since $\ord_z$ has transcendence degree 0. Hence
\begin{equation*}
  \Gamma
  :=\Gamma(V_\bullet)
  :=\{(m,\a)\in\N^{n+1}\mid \a\in\Gamma_m\}
\end{equation*}
is a subsemigroup of $\N^{n+1}$.
Let $\Sigma=\Sigma(V_\bullet)\subset\R^{n+1}$ be the closed convex cone generated by $\Gamma$.
The \emph{Okounkov body} of $V_\bullet$ with respect to $z$ is given by 
\begin{equation*}
  \Delta=\Delta_z(V_\bullet)=\{\a\in\R^n\mid (1,\a)\in\Sigma\}.
\end{equation*}
This is a compact convex subset of $\R^n$.
The Okounkov body of $(X,L)$ is defined as the Okounkov body of $R(X,L)$.

For $m\ge 1$, let $\rho_m$ be the atomic positive measure on $\D$ given by 
\begin{equation*}
  \rho_m=m^{-n}\sum_{\a\in\Gamma_m}\d_{m^{-1}\a}.
\end{equation*}
The following result is a special case of~\cite[Th\'eor\`eme~1.12]{Bou14}.
\begin{thm}\label{T102}
  If $V_\bullet$ contains an ample series, then its Okounkov body $\D\subset\R^n$
  has nonempty interior, and we have $\lim_{m\to\infty}\rho_m=\rho$ 
  in the weak topology of measures, where 
  $\rho$ denotes Lebesgue measure on $\D\subset\R^n$.
  In particular, the limit
  \begin{equation}\label{e119}
    \vol(V_\bullet)
    =\lim_{m\to\infty}\frac{n!}{m^n}\dim_\C V_m
    \in(0,\vol(L)]
  \end{equation}
  exists, and equals $n!\vol(\D)$.
\end{thm}
In fact, the limit in~\eqref{e119} always exists, but may be zero in general;
see~\cite[Th\'eor\`eme~3.7]{Bou14} for a much more precise result due 
to Kaveh and Khovanskii~\cite{KK12}.

For the proof of Theorem~A we will need the following estimate.
\begin{lem}\label{L102}
  For every $\e>0$ there exists $m_0=m_0(\e)>0$ such that 
  \begin{equation*}
    \int_\D g\,d\rho_m \le \int_\D g\,d\rho+\e
  \end{equation*}
  for every $m\ge m_0$ and every concave function 
  $g\colon\D\to\R$ satisfying $0\le g\le 1$.
\end{lem}
The main point here is the uniformity in $g$.
\begin{proof}
Observe that the sets
  \begin{equation*}
    \D_\g:=\{\a\in\R^n\mid \a+[-\g,\g]^n\subset\D\},
  \end{equation*}
  for $\g>0$, form a decreasing family of relatively compact subsets of $\D$ 
  whose union equals the interior of $\D$.
  Since $\partial\D$ has zero Lebesgue measure, we can pick $\g>0$
  such that $\rho(\D\setminus\D_{2\g})\le\e/2$.
  Since $\lim_m\rho_m=\rho$ weakly on $\D$, we get 
  $\varlimsup\rho_m(\D\setminus\D_{\g})\le\rho(\D\setminus\D_{2\g})$,
  so we can pick $m_1$ large enough so that 
  $\rho_m(\D\setminus\D_\g)\le\e$ for $m\ge m_1$.
   Now set $m_0=\max\{m_1,\g^{-1}\}$.  For $m\ge m_0$ we set 
  \begin{equation*}
    A'_m=\{\a\in\tfrac1m\Z^n\mid \a+[0,\tfrac1m]^n\subset\D\}
  \end{equation*}
  and 
  \begin{equation*}
    A_m=\{\a\in\tfrac1m\Z^n\mid \a+[-\tfrac1m,\tfrac1m]^n\subset\D\}.
  \end{equation*}
  If $\la$ denotes Lebesgue measure on the unit cube $[0,1]^n\subset\R^n$, we see that 
  \begin{multline*}
    \int_\D g\,d\rho
    \ge \sum_{\a\in A'_m}\int_{\a+[0,\tfrac1m]^n}g\,d\rho
    = m^{-n}\sum_{\a\in A'_m}\int_{[0,1]^n}g(\a+m^{-1}w)d\la(w)\\
    \ge m^{-n}\sum_{\a\in A'_m}2^{-n}\sum_{w\in\{0,1\}^n}g(\a+m^{-1}w)
    \ge m^{-n}\sum_{\a\in A_m}g(\a)\\
    \ge \int_{\D_\g}g\,d\rho_m
    \ge \int_{\D}g\,d\rho_m-\rho_m(\D\setminus\D_\g)
    \ge \int_{\D}g\,d\rho_m-\e.
  \end{multline*}
  Here the second inequality follows from the concavity of $g$,
  the fourth inequality from the inclusion $A_m\supset\Delta_\gamma\cap\frac1m\Z^n$,
  and the fifth inequality from $g\le 1$.
  This completes the proof.
\end{proof}
%
%
%
%
\subsection{Filtrations}\label{ss:filtrations}
By a \emph{filtration} $\cF$ on $R(X,L)=\bigoplus_mR_m$ we mean the data of a family
\begin{equation*}
  \cF^\la R_m\subset R_m
\end{equation*}
of $\C$-vector subspaces of $R_m$ for $m\in\N$ and $\la\in\R_+$, satisfying
\begin{itemize}
\item[(F1)] 
  $\cF^\lambda R_m\subset\cF^{\lambda'}R_m$ when $\la\ge\la'$; 
\item[(F2)] 
  $\cF^\la R_m=\bigcap_{\la'<\la}\cF^{\la'}R_m$ for $\la>0$;
\item[(F3)] 
  $\cF^0 R_m=R_m$ and  $\cF^\la R_m=0$ for $\la\gg 0$; 
\item[(F4)] 
  $\cF^\la R_m\cdot\cF^{\la'}R_{m'}\subset\cF^{\la+\la'}R_{m+m'}$.
\end{itemize}
The main example for us will be filtrations defined by valuations, see~\S\ref{S101}.
%
%
%
%
\subsection{Induced graded linear series}\label{S104}
Any filtration $\cF$ on $R(X,L)$ defines a family 
\begin{equation*}
  V^t_\bullet
  =V^{\cF,t}_\bullet
  =\bigoplus_mV^t_m
\end{equation*}
of graded linear series of $L$, indexed by $t\in\R_+$, and defined by
\begin{equation*}
  V^t_m:=\cF^{mt}R_m
\end{equation*}
for $m\in\N$.
Set 
\begin{equation*}
  T_m:=T_m(\cF):=\sup\{t\ge0\mid V^t_m\ne0\},
\end{equation*}
with the convention $T_m=0$ if $R_m=0$.
By~(F4) above,
$T_{m+m'}\ge\frac{m}{m+m'}T_m+\frac{m'}{m+m'}T_{m'}$, so
Fekete's Lemma implies that the limit
\begin{equation*}
  T(\cF):=\lim_{m\to\infty}T_m(\cF)\in[0,+\infty]
\end{equation*}
exists, and equals $\sup_mT_m(\cF)$.
By~\cite[Lemma~1.6]{BC11}, $V^t_\bullet$ contains an ample
linear series for any $t<T(\cF)$. It follows that
\begin{equation}\label{e120}
  T(\cF)=\sup\{t\ge 0\mid \vol(V^t_\bullet)>0\}.
\end{equation}
We say that the filtration $\cF$ is \emph{linearly bounded} if $T(\cF)<\infty$.
%
%
%
%
\subsection{Concave transform and limit measure}
Let $\D=\D(L)\subset\R^n$ be the Okounkov body of $R(X,L)$.
The filtration $\cF$ of $R(X,L)$ induces a \emph{concave transform} 
\begin{equation*}
  G=G^\cF\colon\D\to\R_+
\end{equation*}
defined as follows.
For $t\ge 0$, consider the graded linear series $V^t_\bullet\subset R(X,L)$
and the associated Okounkov body $\D^t=\D(V^t_\bullet)\subset\R^n$. 
We have $\D^t\supset\D^{t'}$ for $t<t'$, $\D^0=\D$ and $\D^t=\emptyset$ for $t>T(\cF)$. 
The function $G$ is now defined on $\D$ by 
\begin{equation}\label{e112}
  G(\a)=\sup\{t\in\R_+\mid \a\in\D^t\}.
\end{equation}
In other words, $\{G\ge t\}=\D^t$ for $0\le t\le T(\cF)$.
Thus $G$ is a concave, upper semicontinuous function on $\D$ with values in $[0,T(\cF)]$. 

As noted in the proof of~\cite[Lemma~2.22]{BKMS}, the Brunn-Minkowski 
inequality implies
\begin{prop}\label{p:Brunn}
  The function $t\to\vol(V^t_\bullet)^{1/n}$ is non-increasing and concave on $[0,T(\cF))$.
  As a consequence, it is continuous on $\R_+$, except possibly at $t=T(\cF)$.
\end{prop}
We define the \emph{limit measure} $\mu=\mu^\cF$ of the filtration $\cF$
as the pushforward
\begin{equation*}
  \mu=G_*\rho.
\end{equation*}
Thus $\mu$ is a positive measure on $\R_+$ of mass $\vol(\D)=\frac1{n!}\vol(L)$,
with support in $[0,T(\cF)]$.
\begin{cor}
  The limit measure $\mu$ satisfies
  \begin{equation*}
    \mu
    =-\frac1{n!}\frac{d}{dt}\vol(V^t_\bullet)
    =-\frac{d}{dt}\vol(\D^t)
  \end{equation*}
  and is absolutely continuous with respect to Lebesgue measure,
  except possibly at $t=T(\cF)$, where $\mu\{T(\cF)\}=\lim_{t\to T(\cF)-}\vol(V^t_\bullet)$.
\end{cor}
As a companion to $T(\cF)$ we now define another invariant of $\cF$:
\begin{equation*}
  S(\cF)
  :=\frac1{\vol(L)}\int_0^\infty\vol(V^t_\bullet)\,dt
  =\frac{n!}{\vol(L)}\int_0^\infty t\,d\mu(t)
  =\frac1{\vol(\D)}\int_\D G\,d\rho.
\end{equation*}
Note that $\mu^\cF$, $S(\cF)$, and $T(\cF)$ do not depend on the 
choice of the auxiliary valuation $z$.
\begin{rmk}
  The invariant $S(\cF)$ can also be interpreted as the (suitably normalized) volume
  of the filtered Okounkov body associated to $\cF$, 
  see~\cite[Corollary~1.13]{BC11}.
\end{rmk}
\begin{lem}\label{L106}
  We have $\frac1{n+1}T(\cF)\le S(\cF)\le T(\cF)$.
\end{lem}
\begin{proof}
  The second inequality is clear since $\vol(V^t_\bullet)\le\vol(L)$ and 
  $\vol(V^t_\bullet)=0$ for $t>T(\cF)$. The first follows from the 
  concavity of $t\mapsto\vol(V^t_\bullet)^{1/n}$, which yields
  $\vol(V^t_\bullet)\ge\vol(L)(1-\frac{t}{T(\cF)})^n$.
\end{proof}
\begin{rmk}\label{R101}
  At least when $L$ is ample, a filtration on $R(X,L)$ induces a metric on the Berkovich
  analytification of $L$ with respect to the trivial absolute value on $\C$.
  It is shown in~\cite{trivval} that $S$ and $T$ extend as ``energy-like''
  functionals on the space of such metrics. As a special case of that analysis,
  it is shown that $S(\cF)\le\frac{n}{n+1}T(\cF)$. The case when the filtration is
  associated to a test configuration is treated in~\cite{BHJ1}.
\end{rmk}
%
%
%
%
\subsection{Jumping numbers}
Given a filtration $\cF$ as above, consider the \emph{jumping numbers}
\begin{equation*}
  0\le a_{m,1}\le\dots\le a_{m,N_m}=mT_m(\cF),
\end{equation*}
defined for $m\in M(L)$ by 
\begin{equation*}
  a_{m,j}=a_{m,j}(\cF)=\inf\{\la\in\R_+\mid \codim\cF^\la R_m\ge j\}
\end{equation*}
for $1\le j\le N_m$.
Define a positive measure $\mu_m=\mu_m^\cF$ on $\R_+$ by 
\begin{equation*}
  \mu_m
  =\frac1{m^n}\sum_j\d_{m^{-1}a_{m,j}}
  =-\frac1{m^n}\frac{d}{dt}\dim\cF^{mt}R_m.
\end{equation*}
The following result is~\cite[Theorem~1.11]{BC11}.
\begin{thm}\label{T101}
  If $\cF$ is linearly bounded, \ie~$T(\cF)<+\infty$, then we have 
  \begin{equation*}
    \lim_{m\to\infty}\mu_m=\mu
  \end{equation*}
  in the weak sense of measures on $\R_+$.
\end{thm}

For $m\in M(L)$, consider the rescaled sum of the jumping numbers:
\begin{equation*}
  S_m(\cF)
  =\frac1{mN_m}\sum_ja_{m,j}
  =\frac{m^n}{N_m}\int_0^\infty t\,d\mu_m(t).
\end{equation*}
Clearly $0\le S_m(\cF)\le T_m(\cF)$.
\begin{lem}\label{L105}
  For any linearly bounded filtration $\cF$ on $R(X,L)$ we have 
  \begin{equation}\label{e111}
    S_m(\cF)\le\frac{m^n}{N_m}\int_\D G\,d\rho_m,
  \end{equation}
  for any $m\in M(L)$.
  Further, we have $\lim_{m\to\infty}S_m(\cF)=S(\cF)$.
\end{lem}
\begin{proof}
  The equality $\lim_mS_m(\cF)=S(\cF)$ follows from Theorem~\ref{T101}.
  For the inequality, pick a basis $s_1,s_2,\dots,s_{N_m}$ of $R_m$
  such that $a_{m,j}=\sup \{ \la   \in \R_+ \mid s_j \in \cF^\la R_m \}$ for $1\le j\le N_m$. Set $\a_j:=\ord_z(s_j)$.
  Since $\ord_z$ has transcendence degree 0, we have $\Gamma_m=\{\a_1,\dots,\a_m\}$.
  Thus the right hand side of~\eqref{e111} equals $\frac1{N_m}\sum_{j=1}^{N_m}G(m^{-1}\a_j)$
  whereas the left-hand side is equal to $\frac1{N_m}\sum_{j=1}^{N_m}m^{-1}a_{m,j}$, so it 
  suffices to prove $G(m^{-1}\a_j)\ge m^{-1}a_{m,j}$ for $1\le j\le N_m$. 
  But this is clear from~\eqref{e112}, since $\a_j=\ord_z(s_j)$ and 
  $s_j\in\cF^{a_{m,j}}R_m$ imply $m^{-1}\a_j\in\D^{m^{-1}a_{m,j}}$.
\end{proof}
\begin{cor}\label{C101}
  For every $\e>0$ there exists $m_0=m_0(\e)>0$ such that 
  \begin{equation*}
    S_m(\cF)\le(1+\e)S(\cF)
  \end{equation*}
  for any $m\ge m_0$ and any linearly bounded filtration $\cF$ on $R(X,L)$.
\end{cor}
\begin{proof}
  Set $V:=\vol(\D)$. 
  Pick $\e'>0$ with $(V^{-1}+\e')(V+(n+1)\e')\le(1+\e)$. 
  Note that $0\le G\le T(\cF)$.
  Applying Lemma~\ref{L102} to $g=G/T(\cF)$ we pick 
  $m_0\in M(L)$ such that 
  \begin{equation*}
    \int_\D G\,d\rho_m
    \le\int_\D G\,d\rho+\e'T(\cF)
    =VS(\cF)+\e'T(\cF)
    \le(V+(n+1)\e')S(\cF)
  \end{equation*}    
  for $M(L)\ni m\ge m_0$, 
  where we have used Lemma~\ref{L106} in the last inequality.
  By Theorem~\ref{T102} we may also assume $\frac{m^n}{N_m}\le V^{-1}+\e'$ 
  for $M(L)\ni m\ge m_0$. Lemma~\ref{L105} now yields
  \begin{equation*}
    S_m(\cF)
    \le\frac{m^n}{N_m}\int_\D G\,d\rho_m
    \le(V^{-1}+\e')(V+(n+1)\e')S(\cF)
    \le (1+\e)S(\cF),
  \end{equation*}
  for $M(L)\ni m\ge m_0$, which completes the proof.
\end{proof}
%
%
%
%
\subsection{$\N$-filtrations.}\label{ss:Nfiltrations}
A filtration $\cF$ of $R(X,L)$ is an \emph{$\N$-filtration} if all its jumping numbers are integers, that is,
\begin{equation*}
  \cF^\la R_m=\cF^{\lceil \la \rceil}R_m
\end{equation*}
for all $\la \in \R_+$ and $m \in M(L)$. 
Any filtration $\cF$  induces an $\N$-filtration $\cF_\N$ by setting
\begin{equation*}
  \cF_\N^\la R_m:= \cF^{\lceil \la \rceil} R_m.
\end{equation*}
Note that $\cF_\N$ is a filtration of $R(X,L)$. Indeed, conditions~(F1)--(F3) in~\S\ref{ss:filtrations}
are trivially satisfied and~(F4)
follows from ${\lceil \la \rceil + \lceil \la' \rceil \ge \lceil \la + \la' \rceil}$.

The jumping numbers of $\cF_\N$ and $\cF$ are related by $a_{m,j}(\cF_\N) = \lfloor a_{m,j}(\cF) \rfloor $. 
This implies
\begin{prop}\label{p:truncations}
  If $\cF$ is a filtration of $R(X,L)$, then
  \[
    T_m(\cF_\N) = \lfloor  m\cdot T_m(\cF) \rfloor /m \quad \text{and }
    S_m(\cF)  -m^{-1} \le S_m(\cF_\N) \le S_m(\cF)
  \]
  for $m\in M(L)$. As a consequence, $T(\cF_\N) = T(\cF)$, $S(\cF_\N) = S(\cF)$, 
  and $\mu^{\cF_\N}=\mu^\cF$.
\end{prop}
As a consequence, we obtain the following formula for $S(\cF)$,
similar to~\cite[Lemma~2.2]{FO16}.
\begin{cor}\label{c:Sformula}
  If $\cF$ is a filtration of $R(X,L)$, then
  \begin{equation*}
    S(\cF)
    =S(\cF_\N)
    = \lim_{m \to \infty}\frac1{mN_m}\sum_{j \ge 1} \dim\cF^j R_m.
  \end{equation*}
\end{cor}
\begin{proof}
  Since the jumping numbers of $\cF_\N$ are integers, we have
  \begin{equation*}
    S_m(\cF_\N)
    =\frac1{mN_m}\sum_{j\ge 0}  j\left( \dim\cF_\N^{j} R_m-\dim\cF_\N^{j+1} R_m\right)
    =\frac1{mN_m}\sum_{j \ge 1} \dim\cF_\N^j R_m
  \end{equation*}
  for any $m\in M(L)$.
  Letting $m\to\infty$ and using Proposition~\ref{p:truncations} completes the proof.
\end{proof}
%
%
%
%
%
%
\section{Global invariants of valuations}\label{S105}
As before, $X$ is a normal projective variety of dimension $n$ over $\C$.
Whenever we discuss log discrepancy, $X$ will be assumed to have klt singularities.

Let $L$ be a big line bundle on $X$. Following~\cite{BKMS} we study invariants
of valuations on $X$ defined using the section ring of $L$.
The new results here are Corollary~\ref{C102} and the results in~\S\ref{S108}.
%
%
%
%
\subsection{Induced filtrations}\label{S101}
Any valuation $v\in\Val_X$ induces a filtration $\cF_v$ on $R(X,L)$ via
\begin{equation*}
  \cF^t_vR_m:=\{s\in R_m\mid v(s)\ge t\}
\end{equation*}
for $m\in\N$ and $t\in\R_+$, where we recall that $R_m=H^0(X,mL)$.

We say that $v$ has \emph{linear growth} if $\cF_v$ is linearly bounded.
By Lemma~2.8 in~\cite{BKMS} this notion depends only on $v$ as a valuation,
and not on pair $(X,L)$ (i.e. if $\rho:X' \to X$ is a proper birational morphism with $X'$ normal, 
the condition can be checked on the pair $(X',L')$, where $L'= \rho^*L$). Theorem~2.16 in~\textit{loc.\,cit.} states that if 
$v$ is centered at a closed point on $X$, then $v$ has linear growth iff $\vol(v)>0$.
\begin{lem}
  Any divisorial valuation has linear growth.
  If $X$ has klt singularities, then any $v\in\Val_X$ satisfying $A(v)<\infty$
  has linear growth. 
\end{lem}
\begin{proof}
  We may assume $X$ is smooth. 
  By~\cite[Proposition~2.12]{BKMS}, every divisorial valuation has linear growth. 
  For the second assertion, if $A(v)<\infty$, Izumi's inequality 
  (see~\cite[Proposition~5.10]{jonmus}) 
  implies $v\le A(v)\ord_\xi$, where $\xi=c_X(v)$. Since $\ord_\xi$ is divisorial,
  it has linear growth; hence so does $v$.
\end{proof}
%
%
%
%
\subsection{Global invariants}
Consider a valuation $v$ of linear growth.
We define invariants of $v$ as the corresponding invariants of the induced
filtration $\cF_v$, namely:
\begin{itemize}
\item[(i)]
  the \emph{limit measure} of $v$ is $\mu_v:=\mu^{\cF_v}$;
\item[(ii)]
  the \emph{expected vanishing order} of $v$ is $S(v):=S(\cF_v)=\int_0^\infty t\,d\mu_v(t)$;
\item[(iii)]
  the \emph{maximal vanishing order} or 
  \emph{pseudo-effective threshold} of $v$ is $T(v):=T(\cF_v)$.
\end{itemize}
Note that $T(v)$ is denoted by $a_{{\max}}(\|L\|,v)$ in~\cite{BKMS}.
It follows from Lemma~\eqref{L106} (see also Remark~\ref{R101}) that
\begin{equation}\label{e122}
  \frac1{n+1}T(v)\le S(v)\le T(v).
\end{equation}

The invariants $S$ and $T$ are homogeneous of order 1: $S(tv)=tS(v)$ and 
$T(tv)=tT(v)$ for $t>0$. Similarly, $\mu_{tv}=t_*\mu_v$, where $t\colon\R_+\to\R_+$
denotes multiplication by $t$.
In particular, if $v$ is the trivial valuation on $X$, then $S(v)=T(v)=0$
and $\mu_v=\d_0$. 
\begin{rmk}
  If we think of $v$ as an order of vanishing, then 
  the limit measure $\mu_v$ describes the asymptotic distribution of the 
  (normalized) orders of vanishing of $v$ on $R(X,L)$. This explains the chosen
  name of $S(v)$ and the first name of $T(v)$.
\end{rmk}
For an alternative description of $S(v)$ and $T(v)$, define, for $t\ge 0$, 
\begin{equation*}
  \vol(L; v\ge t)
  :=\vol(V^t_\bullet)
  =\lim_{m\to\infty}\frac{n!}{m^n}\dim\cF^{tm}_vH^0(X,mL).
\end{equation*}
\begin{thm}\label{T103}
  Let $L$ be a big line bundle and 
  $v\in\Val^*_X$ a valuation of linear growth. Then the limit defining 
  $\vol(L;v\ge t)$ exists for every $t\ge 0$. Further:
  \begin{itemize}
  \item[(i)]
    $T(v)=\sup\{t\ge0\mid\vol(L; v\ge t)>0\}$;
  \item[(ii)]
    the function $t\mapsto\vol(L;v\ge t)^{1/n}$ is 
    decreasing and concave on $[0,T(v))$;
  \item[(iii)]
    $\mu_v=-\frac{d}{dt}\vol(L;v\ge t)$; further,
    $\supp\mu_v=[0,T(v)]$, and $\mu$ is absolutely continuous with respect to 
    Lebesgue measure,  except for a possible point mass at $T(v)$;
  \item[(iv)]
    $S(v)=V^{-1}\int_0^{T(v)}\vol(L;v\ge t)\,dt$;
  \item[(v)]
    if $L$ is nef, then the function $t\mapsto\vol(L;v\ge t)$ is strictly decreasing 
    on $[0,T(v)]$ and $\supp\mu_v=[0,T(v)]$.
  \end{itemize}
\end{thm}
\begin{proof}
  The assertions~(i)--(iv) are special cases of the properties of linearly bounded filtrations
  in~\S\ref{S103}. If $L$ is nef, the discussion after Remark~2.7 in~\cite{BKMS}
  shows that $t\mapsto\vol(L;v\ge t)$ is strictly decreasing on $[0,T(v))$. This implies
  $\supp\mu=[0,T(v)]$, so that~(v) holds.
\end{proof}
\begin{rmk}
  In fact, the measure $\mu_v$ likely has no point mass at $T(v)$. 
  This is true when $v$ is divisorial, or simply 
  quasimonomial, see~\cite[Proposition~2.25]{BKMS}.
\end{rmk}

We also define $S_m(v):=S_m(\cF_v)$ and $T_m(v):=T_m(\cF_v)$ for $m\in M(L)$.
These invariants can be concretely described as follows. First,
\begin{equation}\label{e118}
  T_m(v)=\max\{m^{-1}v(s)\mid s\in H^0(X,mL)\}.
\end{equation}
A similar description is true for $S_m$.
\begin{lem}\label{L104}
  For any $m\in M(L)$ and any $v\in\Val_X$ we have 
  \begin{equation}\label{e113}
    S_m(v)
    =\max_{s_j}\frac1{mN_m}\sum_{j=1}^{N_m}v(s_j),
  \end{equation}
  where the maximum is over all bases $s_1,\dots,s_{N_m}$ of $H^0(X,mL)$.
\end{lem}
\begin{proof}
  First consider any basis $s_1,\dots,s_{N_m}$ of $H^0(X,mL)$.
  We may assume $v(s_1)\le v(s_2)\le\dots\le v(s_{N_m})$. 
  Then $v(s_j)\le a_{{m,j}}$, for all $j$, where $a_{m,j}$ is the $j$th jumping number
  of $\cF_vH^0(X,mL)$. 
  Thus $(mN_m)^{-1}\sum_jv(s_j)\le(mN_m)^{-1}\sum_ja_{m,j}=S_m(v)$.
  On the other hand, we can pick the basis such that $v(s_j)=a_{m,j}$,
  and then $(mN_m)^{-1}\sum_jv(s_j)=S_m(v)$.
\end{proof}
Corollary~\ref{C101} immediately implies
\begin{cor}\label{C102}
  For any $v\in\Val_X$ of linear growth, we have $\lim_{m\to\infty}S_m(v)=S(v)$.
  Further, given $\e>0$ there exists $m_0=m_0(\e)>0$ such that 
  if $m\ge m_0$, then 
  \begin{equation*}
    S_m(v)\le S(v)(1+\e)
  \end{equation*}
  for all $v\in\Val_X$ of linear growth.
\end{cor}
%
%
%
%
\subsection{Behavior of invariants}
The invariants $S(v)$, $T(v)$ and $\mu_v$ depend on $L$ (and $X$). If we need to 
emphasize this dependence, we write $S(v;L)$, $T(v;L)$ and $\mu_{v;L}$.
\begin{lem}\label{L107}
  Let $v$ be a valuation of linear growth.
  \begin{itemize}
  \item[(i)]
    If $r\in\N^*$, then $S(v;rL)=rS(v;L)$, $T(v;rL)=rT(v;L)$ and $\mu_{v;rL}=r_*\mu_{v;L}$.
  \item[(ii)]
    If $\rho\colon X'\to X$ is a projective birational morphism, with $X'$ normal, 
    and $L'=\rho^*L$, then $S(v;L')=S(v;L)$, $T(v;L')=T(v;L)$, and $\mu_{v;L'}=\mu_{v;L}$;
  \item[(iii)]
    the invariants $S(v;L)$, $T(v;L)$ and $\mu_{v;L}$ only depend on the 
    numerical class of $L$.
  \end{itemize}
\end{lem}
\begin{proof}
  Properties~(i)--(ii) are clear from the definitions. 
  As for~(iii),~\cite[Proposition~3.1]{BKMS} asserts that the measure $\mu_{v;L}$ only 
  depends on the numerical class of $L$; hence the same true for $S(v;L)$ and $T(v;L)$.
\end{proof}
\begin{rmk}\label{R103}
  In view of~(i) and~(iii) we can define $S(v;L)$ for a big class $L\in\NS(X)_\Q$ by 
  $S(v;L):=r^{-1}S(v;rL)$ for $r$ sufficiently divisible. The same holds for 
  $T(v;L)$ and $\mu_{v;L}$.
\end{rmk}
%
%
%
%
\subsection{The case of divisorial valuations}
We now interpret the invariants $S(v)$ and $T(v)$ in the case when $v$ is a divisorial 
valuation. By homogeneity in $v$ and by Lemma~\ref{L107}~(ii) it suffices to consider the
case when $v=\ord_E$ for a prime divisor $E$ on $X$. 
In this case, $\vol(L;v\ge t)=\vol(L-tE)$, so Theorem~\ref{T103} implies
\begin{cor}\label{C103}
  Let $E\subset X$ be a prime divisor. Then we have:
  \begin{itemize}
  \item[(i)]
    $T(\ord_E)=\sup\{t>0\mid L-tE\ \text {is pseudoeffective}\}$;
  \item[(ii)]
    $S(\ord_E)=\vol(L)^{-1}\int_0^\infty\vol(L-tE)\,dt$.
  \end{itemize}
\end{cor}
  Statement~(i) explains the name pseudoeffective threshold for $T(v)$. 
\begin{rmk}\label{R102}
  The invariants $S(v)$ and $T(v)$ for $v$ divisorial have been explored by 
  K.~Fujita~\cite{Fujitavalcrit}, C.~Li~\cite{Liequivariant}, and Y.~Liu~\cite{Liu16}. 
  In the notation of~\cite{Fujitavalcrit},
  \begin{equation*}
    T(\ord_E)=\tau(E)
    \quad\text{and}\quad
    S(\ord_E)=\tau(E)-\vol(L)^{-1}j(E).
  \end{equation*}
  The invariant $S(\ord_\xi)$, for $\xi\in X$ a regular closed point, also plays an important role 
  in~\cite{MR15} and was used in unpublished work of P.~Salberger from 2006.
\end{rmk}
\begin{prop}\label{P101}
  If $L$ is ample and $v\in\Val_X$ is divisorial, then $\frac1{n+1}\le\frac{S(v)}{T(v)}\le\frac{n}{n+1}$.
\end{prop}
\begin{proof}
  The first inequality follows from the concavity of $t\to\vol(L;v\ge t)^{1/n}$
  and is a special case of Lemma~\ref{L106}. The second inequality is treated
  in~\cite[Proposition~2.1]{Fujitaplt}.
  (In~\textit{loc.\,cit}.\ we have $L=-K_X$, but this assumption is not used in the proof.)
\end{proof}
\begin{rmk}
  When $L$ is ample, Proposition~\ref{P101} in fact holds for any $v\in\Val_X$ of
  linear growth; see Remark~\ref{R101}. 
\end{rmk}
%
%
%
%
\subsection{Invariants as functions on valuation space}\label{S108}
\begin{prop}
  The invariants $S$ and $T$ define lower semicontinuous functions on $\Val_X$.
  For any $m\in M(L)$, the functions $S_m$ and $T_m$ are also lower
  semicontinuous.
\end{prop}
\begin{proof}
  First consider $m\in M(L)$.
  For any nonzero $s\in H^0(X,mL)$, the function $v\mapsto v(s)$ is continuous.
  It therefore follows from~\eqref{e118} and~\eqref{e113} that $S_m$ and $T_m$
  are lower semicontinuous. 
  Hence $T=\sup_mT_m$ is also lower semicontinuous.
  The lower semicontinuity of $S$ is slightly more subtle.
  Pick any $t\in\R_+$. We must show that the set $V:=\{v\in\Val_X\mid S(v)>t\}$
  is open in $\Val_X$. Pick any $v\in V$ and pick $\e>0$ such that 
  $S(v)>(1+\e)t$. By Corollary~\ref{C102}, there exists $m\gg0$ 
  such that $S_m(v)>(1+\e)t$ and $S_m\le(1+\e)S$ on $\Val_X$.
  Since $S_m$ is lower semicontinuous, there exists an open neighborhood $U$ of
  $v$ in $\Val_X$ such that $S_m>(1+\e)t$ on $U$. Then $U\subset V$,
  which completes the proof.  
\end{proof}
\begin{rmk}
  The functions $S$ and $T$ are not continuous in general. Consider the case
  $X=\P^1$, $L=\cO_X(1)$. If $(\xi_j)_{j=1}^\infty$ is a sequence of distinct closed
  points, then $v_j=\ord_{\xi_j}$, $j\ge1$ defines a sequence in $\Val_X$ converging to
  the trivial valuation $v$ on $X$. Then $S(v_j)=1/2$ and $T(v_j)=1$ for all $j$,
  whereas $S(v)=T(v)=0$.
\end{rmk}
The next result is a global version of~\cite[Proposition 2.3]{LX16}.
\begin{prop}\label{p:Sinequality}
  Let $v,w\in\Val_X $ be valuations of linear growth, such that $v\le w$.
  \begin{itemize}
  \item[(i)]
    We have $S(v)\le S(w)$ and $T(v)\le T(w)$.
  \item[(ii)]
    If $L$ is ample and $S(v)=S(w)$, then $v=w$.
  \end{itemize}  
\end{prop}
\begin{rmk}
  The assertion in~(ii) is false for $T$ in general. 
  Indeed, let $X=\P^2$ and $L=\cO_X(1)$.
  Consider an affine toric chart $\A^2\subset\P^2$ with affine coordinates $(z_1,z_2)$.
  Let $v$ and $w$ be monomial valuations in these coordinates with 
  $v(z_1)=w(z_1)=1$ and $0<v(z_2)<w(z_2)\le 1$. 
  Then $w\le v$ and $T(v)=T(w)=1$, but $w\ne v$.
\end{rmk}
\begin{proof}[Proof of Proposition~\ref{p:Sinequality}]
  The assertion in~(i) is trivial. 
  To establish~(ii) we follow the proof of~\cite[Proposition 2.3]{LX16}.
  Note that by Lemma~\ref{L107} we may replace $L$ by a positive multiple.

  Suppose $v\le w$ but $v\ne w$. We must prove $S(v)<S(w)$.
  We may assume there exists $s\in H^0(X,L)$ with $v(s)<w(s)$.
  Indeed, there exists $\la\in\R_+^*$ such that $\fa_\la(v)\subsetneq\fa_\la(w)$.
  Replacing $L$ by a multiple, we may assume $L\otimes\fa_\la(w)$ is globally generated,
  and then 
  \begin{equation*}
    \cF_v^\la H^0(X,L)
    =H^0(X,L\otimes\fa_\la(v))
    \subsetneq H^0(X,L\otimes\fa_\la(w))
    =\cF_w^\la H^0(X,L),
  \end{equation*}
  so that there exists $s\in H^0(X,L)$ with $v(s)<w(s)=\la$.
  After rescaling $v$ and $w$, we may assume 
  $w(s)=p\in\N^*$ and $v(s)\le p-1$.

  We claim that for $m,j \in \N$, we have
  \begin{equation}\label{e115}
    \dim(\cF^j_wR_m/\cF^j_vR_m) 
    \ge \sum_{1 \le i \le \min\{j/p,m\} } \dim  \left(\cF^{j-ip}_v R_{m-i} /\cF^{j-ip+1}_v R_{m-i} \right).
  \end{equation}
  
  To prove the claim, pick, for any $i$ with $1\le i\le\min\{j/p,m\}$, elements 
  \[
    s_{i,1}, \dots,s_{i,b_i}\in \cF^{j-ip}_v R_{m-i} 
  \] 
  whose images form a basis for $\cF^{j-ip}_v R_{m-i} / \cF^{j-ip+1}_v R_{m-i}$.
  As in~\cite[Proposition~2.3]{LX16}, the elements
  \[
    \{
    s^is_{i,l} \mid 1 \le i \le \min\{j/p,m\}, 1 \le l \le b_i
    \}
    \]
  are then linearly independent in $\cF^{j}_{w}R_{m} / \cF^{j}_v R_m$. This completes the proof of the claim.
  
  By Corollary~\ref{c:Sformula} we have
  \begin{equation*}
    S(v)-S(w)
    =\lim_{m \to \infty} \frac{1}{mN_m}\sum_{j\ge 1} \left(\dim\cF_w^jR_m-\dim\cF_v^jR_m\right)\\
  \end{equation*}
  Now~\eqref{e115} gives
  \begin{align*}
    \sum_{j\ge 1} \left(\dim\cF_w^jR_m-\dim\cF_v^jR_m\right)
    &\ge
      \sum_{j\ge1} \sum_{1\le i \le \min\{\frac{j}{p},m\}}\left(\dim\cF_v^{j-ip}R_{m-i}-\dim\cF_v^{j-ip+1} R_{m-i}\right)\\
    &=\sum_{1\le i\le m} \sum_{j\ge pi}\left(\dim\cF_v^{j-ip}R_{m-i}-\dim\cF_v^{j-ip+1} R_{m-i}\right)\\
    & =\sum_{1 \le i \le m} \dim R_{m-i}
  \end{align*}
  We conclude that
  \begin{equation*}
    S(v)-S(w) 
    \ge 
    \limsup\limits_{m \to \infty} \frac{1}{mN_m} \sum_{1 \le i \le m} \dim(R_{m-i}) >0,
  \end{equation*}
  since $\dim R_m=N_m\sim m^n(L^n)$ as $m\to\infty$.
  This completes the proof.
\end{proof}
%
%
%
%
\subsection{Base ideals of filtrations}\label{S114}
In this section we assume $L$ is \emph{ample}.
To an arbitrary filtration $\cF$ of $R(X,L)$ we associate
\emph{base ideals} as follows.
For $\la\in\R_+$ and $m\in M(L)$, set 
\begin{equation*}
  \fb_{\la,m}(\cF) := \bs{\cF^\la H^0(X,mL)}.
\end{equation*}
\begin{lem}
  For $\la\in\R_+$ the sequence $(\fb_{\la,m}(\cF))_m$ is stationary (i.e. $\fb_{\la,m} = \fb_{\la,m+1}$ for $m\in M(L)$ sufficiently large , 
  with limit $\sum_{m\in M(L)}\fb_{\la,m}$.
\end{lem}
\begin{proof}
  It follows from~(F4) that if $m_1,m_2\in M(L)$ and $\la_1,\la_2\in\R_+$, then
  \begin{equation}\label{e125}
    \fb_{\la_1,m_1}(\cF) \cdot \fb_{\la_2,m_2}(\cF)\subset \fb_{\la_1+\la_2, m_1+m_2}(\cF)
  \end{equation}
  Since $L$ is ample, there exists 
  $m_0\in\N^*$ such that $mL$ is globally generated for $m\ge m_0$.
  In particular, $\fb_{0,m}=\cO_X$ for $m\ge m_0$.
  As a consequence of~\eqref{e125}, if $m\in M(L)$ and 
  $m'\ge m_0$, then 
  $\fb_{\la,m+m'}(\cF)\supset\fb_{\la,m}(\cF)\cdot\fb_{0,m'}(\cF)=\fb_{\la,m}(\cF)$. 
  The lemma follows.
\end{proof}
Using the lemma, set $\fb_\la(\cF):=\fb_{\la,m}(\cF)$ for $m\gg0$.
Thus $\fb_{\la,m}(\cF)\subset\fb_\la(\cF)$ for $m\in M(L)$.
\begin{cor}
  We have $\fb_0(\cF)=\cO_X$ and 
  $\fb_\la(\cF)\cdot\fb_{\la'}(\cF)\subset\fb_{\la+\la'}(\cF)$ for $\la,\la'\in\R_+$.
  In particular, the sequence $(\fb_p(\cF))_{p\in\N^*}$ is a graded sequence of ideals.
\end{cor}
\begin{lem}\label{l:basevalideals}
  If $v$ is a valuation on $X$, then $\fb_\la(\cF_v)= \fa_\la(v)$ for all $\la\in\R_+$.
\end{lem}
\begin{proof}
  Given $\la$, $mL\otimes \fa_\la(v)$ is globally generated for $m\gg0$; 
  hence $\fb_{\la,m}(\cF_v)=\fa_\la(v)$.  
\end{proof}
Using base ideals, we can relate the invariants of a filtration to those of a valuation. 
\begin{lem}\label{L108}
  If $v(\fb_\bullet(\cF))\ge1$, then $\cF^p R_m \subset \cF^p_{v} R_m$ 
  for all $m\in M(L)$ and $p\in\N^*$.
\end{lem}
\begin{proof}
  We have $1\le v(\fb_\bullet(\cF)) \le v(\fb_p(\cF))/p$.
  Thus $v(\fb_p(\cF))\ge p$, so that $\fb_p(\cF) \subset \fa_p(v)$. 
  Since we also have $\fb_{\la,m}(\cF)\subset\fb_\la(\cF)$ for all $m\in M(L)$, 
  this implies
  \[
    \cF^p R_m 
    \subset H^0 (X, mL\otimes \fb_{p,m}(\cF)) 
    \subset H^0 (X, mL\otimes \fa_p(v)) 
    =\cF_v^pR_m.
  \]
  which completes the proof.
\end{proof}
\begin{cor}\label{c:vFinequality}
  Let $\cF$ be a linearly bounded filtration of $R(X,L)$. Then 
  \begin{equation*}
    S(v)\ge v(\fb_\bullet(\cF))S(\cF)
      \quad\text{and}\quad
      T(v)\ge v(\fb_\bullet(\cF))T(\cF),
    \end{equation*}
    for any valuation $v\in\Val_X$.
  \end{cor}
\begin{proof}
  The assertions are trivial when $v(\fb_\bullet(\cF))=0$, so 
  we may assume $v(\fb_\bullet(\cF))=1$ after scaling $v$.
  In this case, Lemma~\ref{L108} shows
  that $\cF^p R_m \subset \cF^p_{v} R_m$ for $p\in\N^*$ and $m\in M(L)$. 
  Using Proposition~\ref{p:truncations} and Corollary~\ref{c:Sformula}, 
  this implies 
  \begin{equation*}
    S(\cF)
    =S(\cF_\N)
    \le S(\cF_{v,\N})
    = S(\cF_v)
    =S(v),
  \end{equation*}
  and similarly $T(\cF) \le T(v)$. The proof is complete.
\end{proof}
%
%
%
%
%
%
\section{Thresholds}\label{S106}
Let $X$ be a normal projective variety with klt singularities, 
and $L$ a big line bundle on $X$. In this section we study the log-canonical
threshold of $L$, and introduce a new related invariant, the stability threshold of $L$.
Both are defined in terms of the asymptotic behavior of the singularities 
of the members of the linear system $|mL|$ as $m\to\infty$.
%
%
%
%
\subsection{The log canonical threshold}
Following~\cite{CS08} the \emph{log canonical threshold} $\a(L)$ 
of $L$ is the infimum of $\lct(D)$ with $D$ an effective $\Q$-divisor 
$\Q$-linearly equivalent to $L$. 
As explained by Demailly (see~\cite[Theorem~A.3]{CS08}),
this can be interpreted analytically as a generalization of the $\alpha$-invariant 
introduced by Tian~\cite{Tian97}. 

For $m\in M(L)$, we also set
\begin{equation*}
  \a_m(L):=\inf \{ m \lct(D)\mid D \in |mL| \}.
\end{equation*}
It is then clear that $\a(L)=\inf_{m\in M(L)}\a_m(L)$. The invariants 
$\a_m$ and $\a$ can be computed using invariants of valuations, as follows:
\begin{prop}
  For $m\in M(L)$, we have 
  \begin{equation}\label{e107}
     \a_m(L)
     =\inf_v\frac{A(v)}{T_m(v)}
     =\inf_E\frac{A(\ord_E)}{T_m(\ord_E)},
  \end{equation}
  where $v$ runs through nontrivial valuations on $X$ with $A(v)<\infty$,
  and $E$ through prime divisors over~$X$.
\end{prop}
\begin{proof}
  Writing out the definition of $\lct(D)$, we see that
  \[
    \a_m(L) = m \cdot  \inf_{D\in |mL|} \left(\inf_{v} \frac{A(v)}{v(D)}\right),
  \]
  where the second infimum may be taken over nontrivial valuations with finite log discrepancy,
  or only divisorial valuations. 
  Switching the order of the two infima and noting 
  $\sup_{D\in |mL|} v(D)  = m \cdot T_m(v)$ yields~\eqref{e107}.
\end{proof}
\begin{cor}\label{C105}
  We have 
  \begin{equation}\label{e108}
    \a(L)
     =\inf_v\frac{A(v)}{T(v)}
     =\inf_E\frac{A(\ord_E)}{T(\ord_E)},
  \end{equation}
  where $v$ runs through valuations on $X$ with $A(v)<\infty$
  and $E$ over prime divisors over~$X$.
\end{cor}
\begin{proof}
  Since $T(v) = \sup_{m\in M(L)} T_m(v)$,~\eqref{e108} follows from~\eqref{e107}.
\end{proof}
%
%
%
%
\subsection{The stability threshold}
Given $m\in M(L)$, we say, following~\cite{FO16}, that an effective $\Q$-divisor $D\sim_\Q L$ 
is of \emph{$m$-basis type} if there exists a basis $s_1,\dots,s_{N_m}$
of $H^0(X,mL)$ with
\begin{equation}\label{e110}
  D=\frac1{mN_m}\sum_{j=1}^{N_m}\{s_j=0\}.
\end{equation}
Set
\begin{equation}\label{e109}
  \d_m(L):=\inf\{\lct(D)\mid\ \text{$D$ of $m$-basis type}\},
\end{equation}
and define the \emph{stability threshold} of $L$ as 
\begin{equation*}
  \d(L) := \limsup \limits_{m\to\infty} \d_m(L).
\end{equation*}
We shall see shortly that this limsup is in fact a limit.
\begin{prop}\label{p:deltam}
  For $m \in M(L)$, we have 
  \begin{equation*}
    \d_m(L)
     =\inf_v\frac{A(v)}{S_m(v)}
     =\inf_E\frac{A(\ord_E)}{S_m(\ord_E)},
   \end{equation*}
  where $v$ runs through nontrivial valuations on $X$ with $A(v)<\infty$
  and $E$ through prime divisors over~$X$.
\end{prop}
\begin{proof}
  Note that 
  \begin{equation*}
    \d_m(L) =  \inf_{D\ \text{of $m$-basis type}} \left(\inf_{v} \frac{A(v)}{v(D)} \right),
  \end{equation*}
  where the second infimum runs through all valuations with $A(v)< \infty$ 
  or only divisorial valuations of the form $v=\ord_E$. 
  Switching the order of the two infima and applying Lemma~\ref{L104} 
  yields the desired  equality. 
\end{proof}

\begin{thm}\label{t:delta}
  We have $\d(L)=\lim_{m\to\infty}\d_m(L)$.
  Further, 
  \begin{equation*}
    \d(L) 
    =\inf_{v} \frac{A(v)}{S(v)}
     =\inf_E\frac{A(\ord_E)}{S(\ord_E)},
   \end{equation*}
  where $v$ runs through nontrivial valuations on $X$ with $A(v)<\infty$
  and $E$ through prime divisors over~$X$.
\end{thm}
\begin{proof}
  We will only prove the first equality; the proof of the second being essentially identical.
  Let us use Proposition~\ref{p:deltam} and Corollary~\ref{C102}.
  The fact that $\lim_{m\to\infty}S_m=S$ pointwise on $\Val_X$ directly shows that 
  \begin{equation}\label{e114}
    \limsup_m\d_m(L)
    \le\inf_v\frac{A(v)}{S(v)}.
  \end{equation}
  On the other hand, given $\e>0$ there exists $m_0=m_0(\e)$ such that 
  $S_m(v)\le(1+\e)S(v)$ for all $v\in\Val_X$ and $m\ge m_0$.
  Thus 
  \begin{equation*}
    \liminf_m\d_m(L)
    =\liminf_m\inf_v\frac{A(v)}{S_m(v)}
    \ge(1+\e)^{-1}\inf_v\frac{A(v)}{S(v)}.
  \end{equation*}
  Letting $\e>0$ and combining this inequality with~\eqref{e114} completes the proof.
\end{proof}
\begin{rmk}\label{R104}
  It is clear that $\a(rL)=r^{-1}\a(L)$ and $\d(rL)=r^{-1}\d(L)$ for any $r\in\N^*$.
  This allows us to define $\a(L)$ and $\d(L)$
  for any big $\Q$-line bundle $L$, by setting 
  $\a(L):=r^{-1}\a(rL)$ and  $\d(L):=r^{-1}\d(rL)$ for $r$ sufficiently divisible.
\end{rmk}
%
%
%
%
\subsection{Proof of Theorems~A,~B and~C}
We are now ready to prove the first three main results in the introduction.

We start with Theorems~A and~C. 
The existence of the limit $\d(L)=\lim_m\d_m(L)$ was proved above,
so Theorem~C follows immediately from Corollary~\ref{C105} and Theorem~\ref{t:delta}.
Let us prove the remaining assertions in Theorem~A.

The estimate $\a(L)\le\d(L)\le(n+1)\a(L)$ follows from the corresponding 
inequalities in~\eqref{e122} between $T(v)$ and $S(v)$
together with Theorem~C.
When $L$ is ample, we obtain the stronger inequality $\d(L)\ge\frac{n+1}{n}\a(L)$
using Proposition~\ref{P101}.
The fact that $\a(L)$ and $\d(L)$ only depend on the numerical equivalence class
of $L$ follows from the corresponding properties of the invariants $S(v)$ and $T(v)$,
see Lemma~\ref{L107}~(iii). 
Finally we prove that $\a(L)$ and $\d(L)$ are strictly positive. It suffices to consider
$\a(L)$. The case when $L$ is ample is handled in~\cite[Theorem~9.14]{BHJ1}
using Seshadri constants, and the general case follows from Lemma~\ref{l:alphamonotone} 
below by choosing $D$ effective such that $L+D$ is ample.

\begin{lem}\label{l:alphamonotone}
  If $L$ is a big line bundle and $D$ is an effective divisor, then $\a(L+D)\le\a(L)$.
\end{lem}

The statement is already in the literature~\cite[Lemma 4.1]{Der15}. 
We provide a proof for the convenience of the reader.

\begin{proof}
  Given $m\in M(L)$, the assignment $F\mapsto F+mD$ defines an injective map
  from $|mL|$ to $|m(L+D)|$. Since $\lct(F+mD)\le\lct(F)$ for all $F\in|mL|$, 
  it follows that $\a_m(L+D)\le\a_m(L)$. Letting $m\to\infty$ completes the proof.
\end{proof}

\smallskip
Finally we prove Theorem~B, so suppose $X$ is a $\Q$-Fano variety. 
The argument relies heavily on the work by K.~Fujita and C.~Li, who exploited ideas from
the Minimal Model Program, as adapted to K-stability questions by 
C.~Li and C.~Xu~\cite{LX14}.

First assume $K_X$ is Cartier.
By either~\cite[Theorem~3.7]{Liequivariant} or~\cite[Corollary~1.5]{Fujitavalcrit}, 
$X$ is K-semistable iff
$\b(E)\ge0$ for all prime divisors $E$ over $X$. In our notation, this reads 
$A(\ord_E)\ge S(\ord_E)$ for all $E$, see~\cite[Definition~1.3~(4)]{Fujitavalcrit}
and Remark~\ref{R102}, and is hence equivalent to $\delta(-K_X)\ge1$ 
in view of Theorem~\ref{t:delta}.

Similarly, by~\cite[Corollary~1.5]{Fujitavalcrit}, $X$ is uniformly $K$-stable iff there exists $\e>0$ such that 
$\b(E)\ge\e j(E)$ for all divisors $E$ over $X$. This reads
$A(\ord_E)-S(\ord_E)\ge\e(T(\ord_E)-S(\ord_E))$ for all $E$.
Since $-K_X$ is ample, Proposition~\ref{P101} implies 
$n^{-1}S(\ord_E)\le T(\ord_E)-S(\ord_E)\le nS(\ord_E)$, so $X$ is 
uniformly K-stable iff there exists $\e'>0$ such that 
$A(\ord_E)-S(\ord_E)\ge \e'S(\ord_E)$ for all $E$. But this is equivalent to
$\d(-K_X)>1$ by Theorem~\ref{t:delta}.

When $K_X$ is merely $\Q$-Cartier, the argument is similar, 
using Lemma~\ref{L107}; see Remark~\ref{R104}.
%
%
%
%
\subsection{Volume estimates}\label{S111}
We now prove Theorem~D, giving a lower bound on the volume of $L$.
This theorem is a consequence of the following proposition, 
first observed by Liu, and embedded in the proof of~\cite[Theorem 21]{Liu16}. 
\begin{prop}
  If $v\in\Val_X^*$ has linear growth and is centered at a closed point, then 
  \[
    T(v)\ge \sqrt[n]{\vol(L)/\vol(v)}
    \quad\text{and}\quad
    S(v) \ge \frac{n}{n+1} \sqrt[n]{\vol(L)/\vol(v)}.
  \]
\end{prop}
\begin{proof}
  We follow Liu's argument. 
  By the exact sequence
  \[
    0  \to H^0(X, mL \otimes \fa_{mt}(v)) \to H^0(X,mL) \to H^0(X,mL \otimes ( \cO_X / \fa_{mt}(v)),
  \]
  we see that
  \[
    \dim \cF_v^{mt} H^0(X,mL) \ge \dim H^0(X,mL) - \ell(\cO_{X,\xi}/ \fa_{mt}(v)),
  \]
  where $\xi\in X$ is the center of $v$.
  Diving by $m^n/ n!$ and taking the limit as $m \to \infty$ gives 
  \[
    \vol(L; v\ge t ) \ge \vol(L) - t^n \vol(v),
  \]
  which implies the lower bound for $T(v)$. 
  Further, integrating with respect to $t$ shows that
  \begin{align*}
    S(v) 
    & = \frac{1}{\vol(L)} \int_0^{T(v) } \vol(L ; v\ge t) \, dt \\
    & \ge \frac{1}{\vol(L)} \int_0 ^{\sqrt[n]{\vol(L) /\vol(v)}} ( \vol(L) - t^n \vol(v) )\, dt\\
    & = \frac{n}{n+1} \sqrt[n]{\vol(L)/\vol(v)},
  \end{align*}
  which completes the proof.
\end{proof}

\begin{proof}[Proof of Theorem~D]
  If $A(v)= \infty$, then $\widehat{\vol}(v)=\infty$ and the inequality is trivial. 
  If $A(v) < \infty$, then $v$ has linear growth and the previous proposition gives
  \[
    \vol(L) \le \left( \frac{n+1}{n}\right)^n S(v)^n \vol(v) = 
    \left( \frac{n+1}{n}\right)^n \left( \frac{S(v)}{A(v)} \right)^n \widehat{\vol}(v).
  \]
  Since $\d(L) \le A(v)/S(v)$ by Theorem~\ref{t:delta}, the proof is complete. 
\end{proof}
%
%
%
%
\subsection{Valuations computing the thresholds}\label{ss:propvalcomputing}
We say that a valuation $v\in\Val_X^*$ with $A(v)<\infty$ \emph{computes} the 
log-canonical threshold (resp.\ the stability threshold) of $L$ if $\a(L)=A(v)/T(v)$
(resp.\ $\d(L)=A(v)/S(v))$. In~\S\ref{s:theoremE} we will prove that such valuations 
always exist when $L$ is ample. Here we will describe some general properties
of valuations computing one of the two thresholds.

We start by the following general result.
\begin{prop}\label{p:adcompute}
  Let $v$ be a nontrivial valuation on $X$ with $A(v)<\infty$.
  \begin{itemize}
  \item[(i)]
    if $v$ computes $\a(L)$ or $\d(L)$, then $v$ computes $\lct(\ab(v))$;
  \item[(ii)]
    if $L$ is ample and $v$ computes $\d(L)$, then $v$ is the unique valuation,
    up to scaling, that computes $\lct(\ab(v))$.
  \end{itemize}
\end{prop}
\begin{proof}
  First suppose $v\in \Val_X$ computes $\a(L)$.
  Recall that $\lct(\ab(v))=\inf_w\frac{A(w)}{w(\ab(v))}$, where it suffices to consider the 
  infimum over $w\in\Val_X^*$ normalized by $w(\ab(v))=1$.
  The latter condition implies $w(\fa_p(v))\ge p$ for all $p$, so that $w\ge v$. 
  By Proposition~\ref{p:Sinequality}~(i), this yields $T(w) \ge T(v)$.
  Since $v$ computes $\a(L)$, we have $A(w)/T(w)\ge A(v)/T(v)$.
  Thus
  \begin{equation*}
    A(v)/v(\ab(v))
    =A(v)
    \le A(w)
    =A(w)/w(\ab(v)),
  \end{equation*}
  so taking the infimum over $w$ shows that $v$ computes $\lct(\ab(v))$. 
  The case when $v$ computes $\d(L)$ is handled in the same way, 
  and the uniqueness statement in~(ii) follows from Proposition~\ref{p:Sinequality}~(ii).
\end{proof}

\begin{conj}\label{conj:qm}
  Any valuation computing $\a(L)$ or $\d(L)$ must be quasimonomial.
\end{conj}
Note that the strong version of Conjecture~B in~\cite{jonmus} implies Conjecture~\ref{conj:qm} in view of 
Proposition~\ref{p:adcompute}.

While Conjecture~\ref{conj:qm} seems difficult in general, it is trivially true in dimension one 
(since all valuations are then quasimonomial). We also have
\begin{prop}\label{p:twodim}
  If $X$ is a projective surface with at worst canonical singularities, then:
  \begin{itemize}
  \item[(i)]
    any valuation computing $\a(L)$ or $\d(L)$ must be quasimonomial;
  \item[(ii)]
    if $X$ is smooth, then any valuation computing $\a(L)$ or $\d(L)$ 
    must be monomial in suitable local coordinates at its center.
  \end{itemize}
\end{prop}
We expect that the statement in~(i) holds for klt surfaces as well.
\begin{proof}
  Suppose $v\in\Val_X^*$ computes $\a(L)$ or $\d(L)$.
  By Proposition~\ref{p:adcompute}, $v$ computes $\lct(\ab(v))$. 
  Let $Y\to X$ be a resolution of singularities of $X$. Since $X$ has canonical
  singularities, the relative canonical divisor $K_{Y/X}$ is effective, and $v$
  also computes the jumping number $\lct_Y^{K_{Y/X}}(\ab(v))$. 
  By~\cite[\S9]{jonmus}, $v$ is quasimonomial, proving~(i).

  The statement in~(ii) follows from~\cite[Lemma~2.11~(i)]{valmul}.
\end{proof}

\begin{rem}\label{deltaconj}
Since the first version of this paper, it was shown by Xu that a weak version of~\cite[Conjecture B]{jonmus} holds; see~\cite[Theorem 1.1]{Xu19}.  
Combining the result in~\textit{loc.\,cit} with Proposition~\ref{p:adcompute}.ii gives that any valuation computing $\delta(L)$ is quasimonomial. 
\end{rem}

Finally we consider the case of \emph{divisorial} valuations computing one of the 
two thresholds.
In~\cite{Blu16a}, the author studied properties of divisorial valuations that compute log canonical thresholds of graded sequences of ideals. 
The following proposition follows from Proposition~\ref{p:adcompute} 
and results in~\cite{Blu16a}.
\begin{prop}\label{p:pltblowup}
  Let $v$ be a divisorial valuation on $X$.
     \begin{itemize}
  \item[(i)]
    If $v$ computes $\a(L)$ or $\d(L)$, then there exists a prime divisor $E$ over $X$ 
    of log canonical type such that $v= c \ord_E$ for some $c\in \R_+$.  
  \item[(ii)]
    If $v$ computes $\d(L)$ and $L$ is ample, then there exists a prime divisor $E$ over $X$ 
    of plt type such that $v= c \ord_E$ for some $c\in \R_+$.   
  \end{itemize}
\end{prop}
We explain some of the above terminology. 
Let $E$ be a divisor over $X$ such that there exists a projective birational morphism $\pi\colon Y \to X$ such that $E$ is a prime divisor on $Y$ and $-E$ is $\Q$-Cartier and $\pi$-ample. 
We say that $E$ is of \emph{plt} (resp., \emph{log canonical}) \emph{type} if the pair $(Y,E)$ is plt (resp., log canonical)~\cite[Definition 1.1]{Fujitaplt}. 
K. Fujita considered plt type divisors in~\cite{Fujitaplt}. Note that Proposition~\ref{p:pltblowup} (ii) is similar to results in~\cite{Fujitaplt}. 
\begin{proof}
  We may assume $v= \ord_F$ for a divisor $F$ over $X$. If $v$ computes $\a(L)$ or $\d(L)$, 
  then we may apply Proposition~\ref{p:adcompute}~(i) to see $A(v) = \lct( \ab(v))$. 
  Furthermore, if $v$ computes $\d(L)$ and $L$ is ample, 
  Proposition~\ref{p:adcompute}~(ii)
  implies $A(v) < A(w)/w(\ab(v))$ as long as $w$ is not a scalar multiple of $v$. 
  The statement now follows from Propositions~1.5 and~4.4 of~\cite{Blu16a}.  
\end{proof}
%
%
%
%
\section{Uniform Fujita approximation}\label{s:fujitaapprox}
In this section we prove Fujita approximation type statements for filtrations arising from 
valuations.\footnote{The term Fujita approximation refers to the work of~T.~Fujita~\cite{Fuj94}.}
These results play a crucial role in the proof of Theorem~E.

Related statements have appeared in the literature. See~\cite[Theorem D]{LM09} 
for the case of graded linear series and~\cite[Theorem~1.14]{BC11} for the case of 
filtrations. Here we specialize to filtrations defined by valuations, and the main
point is to have uniform estimates in terms of the log discrepancy of the
valuation. To this end we use multiplier ideals.

Throughout this section, $X$ is a normal projective $n$-dimensional klt variety.
%
%
%
%
\subsection{Approximation results}
 Given a valuation $v$ on $X$ and a line bundle $L$ on $X$, we seek to understand 
 how well $S(v)$ and $T(v)$ can be 
 approximated by studying the filtration $\cF_v$ restricted to $H^0(X,mL)$ for $m$ large but fixed.

 Recall that the pseudoeffective threshold of $v$ is defined by
 $T(v):=\lim_{m \to \infty} T_m(v)$. 
 \begin{thm}\label{t:fujitaT}
  Let $X$ be a normal projective klt variety and $L$ an ample line bundle on $X$.
  Then there exists a constant $C=C(X,L)>0$ such that 
  \[
    0\le T(v)-T_m(v) \le\frac{CA(v)}{m}
  \]
  for all $m\in M(L)$ and all $v\in\Val_X^*$ with $A(v)<\infty$.
\end{thm}
\begin{cor}
  We have
  $0\le\a(L)^{-1} -\a_m(L)^{-1} \le\frac{C}{m}$ for all $m\in M(L)$.
\end{cor}
We also have a version of Theorem~\ref{t:fujitaT} for the expected order of 
vanishing $S(v)$,  but this is in terms of a modification $\tilde{S}_m(v)$ of the 
invariant $S_m(v)$, which we first need to introduce.

Let $V_\bullet$ be a graded linear series of a line bundle $L$ on $X$. For $m\in \N^*$, we write $V_{m,\bullet}$ for the graded linear series of $mL$ defined by 
\begin{equation*}
  V_{m,\ell} := H^0(X, m \ell L \otimes  \overline{\fa^\ell}) \subset H^0(X, m \ell L),
\end{equation*}
where $\fa$ denotes the base ideal  $\bs{V_m}$ and $\overline{\fa^\ell }$ the integral closure 
of the ideal $\fa^\ell$.

 If $V_{m}= 0$, then it is clear that $V_{m,\ell}=0$ for all $\ell \in \N^*$ and $\vol(V_{m,\bullet})= 0$. 
When $V_{m}\neq 0$, we use the geometric characterization of the integral closure as in~\cite[Remark 9.6.4]{LazPAG} to express $V_{m,\ell}$ as follows. 
Let $\mu\colon Y_m \to X$ be a proper birational morphism such that 
$Y_m$ is normal and $\bs{V_m} \cdot \cO_Y = \cO_Y(-F_m)$ for some effective Cartier 
divisor $F_m$. Then 
\begin{equation*}
  V_{m, \ell} \simeq H^0(Y_m,\ell(m \mu^*(L)-F_m) ) 
\end{equation*}
for all $\ell\ge1$.
Since $m \mu^*(L)-F_m$ is base point free and therefore nef, 
\begin{equation*}
  \vol(V_{m,\bullet}) =((m\mu^*(L)-F_m)^n)
\end{equation*}
by~\cite[Corollary 1.4.41]{LazPAG}.

In the case when $V_\bullet$ contains an ample series, we have
\[
\vol(V_{\bullet}) = \lim_{m \to \infty} \frac{\vol(V_{m,\bullet})}{m^n};
\]
see~\cite[Proposition~17]{His13} and also~\cite[Appendix]{Sze15}.

Now consider a filtration $\cF$ of $R(X,L)$. As in~\S\ref{S104}, this gives rise
to a family $V_{m}^t=V_{m}^{\cF,t}$ of graded linear series of $L$, 
indexed by $t\in\R_+$, and defined by
\begin{equation*}
  V_{m}^t:=\cF^{mt}R_m.
\end{equation*}
Using the previously defined notion, we get an additional family of graded linear series 
$V_{m, \bullet}^t$ of $mL$ for each $m\in \N^*$. Specifically, 
\[
  V_{m,\ell}^t := H^0(X, m\ell L\otimes\overline{\bs{V_m^t}^\ell}).
\]
Clearly $\vol(V_{m,\bullet}^t)$ is a decreasing function of $t$ that vanishes 
for $t>T(\cF)$. When $\cF$ is linearly bounded, we write 
\begin{equation*}
  \tilde{S}_m(\cF) := \frac1{m^n\vol(L)}\int_0^{T(\cF)} \vol\left( V_{m, \bullet}^t\right)\, dt.
\end{equation*}
Note that by the dominated convergence theorem, 
\[
  S(\cF) = \lim_{m \to \infty} \tilde{S}_m(\cF).
\]
When $v$ is a valuation on $X$ with linear growth, we set $\tilde{S}_m(v) := \tilde{S}_m(\cF_v)$. 

\begin{thm}\label{t:fujitaS}
  Let $X$ be a normal projective klt variety and $L$ an ample line bundle on $X$.
  Then there exists a constant $C=C(X,L)$ such that 
  \[
    0\le S(v)-\tilde{S}_m(v) \le\frac{CA(v)}{m}
  \]
  for all $m\in \N^*$ and all $v\in\Val_X$ with $A_X(v)<\infty$.
\end{thm}
Theorems~\ref{t:fujitaT} and~\ref{t:fujitaS} may be viewed as global analogues 
of~\cite[Proposition 3.7]{Blu16b}.
Their proofs, which appear at the end of this section, use 
multiplier ideals and take inspiration from~\cite{DEL} and~\cite{ELS}.
%
%
%
%
 \subsection{Multiplier ideals}
For an excellent reference on multiplier ideals, see~\cite{LazPAG}.

 Let $\fa$ be a nonzero ideal on $X$. 
 Consider a log resolution $\mu\colon Y \to X$ of $\fa$, and write 
 $\fa\cdot \cO_Y = \cO_Y(-D)$. 
 For $c\in\Q_+^*$, the \emph{multiplier ideal} $\cJ(X, c \cdot \fa)$ is defined by 
\[
  \cJ(X, c\cdot \fa) :=  \mu_\ast \cO_Y\left(  \lceil K_{Y/X} -  cD \rceil \right) \subset \cO_X.
\]
It is a basic fact that the multiplier ideal is independent of the choice of $\mu$.

If $c\in\N^*$, then $\cJ(X,c\cdot \fa) = \cJ(X,\fa^c)$. 
We will use the convention that $\cJ(X, c \cdot (0)  ) := (0)$, where $(0) \subset \cO_X$ denotes the zero ideal.

Multiplier ideals satisfy the following containment relations. 
See~\cite[Proposition 9.2.32]{LazPAG} for the case when $X$ is smooth.
\begin{lem}\label{l:mult}
  Let $\fa,\fb$ be nonzero ideals on $X$.
  \begin{enumerate}
  \item We have $\fa\subset \cJ(X, \fa)$.
  \item If $\fa \subset \fb$ and $c>0$ a rational number, then
    $\cJ(X,c\cdot \fa) \subset \cJ(X,c\cdot \fb)$.
  \item If $c\ge d>0$ are rational numbers, then 
    $\cJ(X, c \cdot \fa) \subset \cJ(X, d\cdot \fa)$.
  \end{enumerate}
\end{lem}
The following \emph{subadditivity} theorem was proved by Demailly, Ein, and Lazarsfeld 
in the smooth case~\cite{DEL}. 
The case below was proved by Takagi~\cite[Theorem 2.3]{Tak06} and, later, by Eisenstein~\cite[Theorem 7.3.4]{Eis11}.

\begin{thm}\label{t:sub}
  If $\fa,\fb$ are nonzero ideals on $X$, and $c\in\Q_+^*$, then 
  \begin{equation*}
    \Jac_X \cdot 
    \cJ(X, c \cdot( \fa \cdot \fb))\subset \cJ(X, c \cdot \fa)\cdot \cJ(X, c\cdot \fb),
  \end{equation*}
  where $\Jac_X$ denotes the Jacobian ideal as defined in~\cite[p.~402]{Eisenbud}. 
\end{thm}

%
%
%
%
\subsection{Asymptotic multiplier ideals}
Let $\ab$ be a graded sequence of ideals on $X$ and $c>0$ a rational number. 
By Lemma~\ref{l:mult}, we have  
\[
\cJ \left( X, (c/p) \cdot \fa_{p} \right) \subset \cJ \left( X, c/(pq) \cdot \fa_{p q} \right) 
\]
for all positive integers $p,q$. This, together with the Noetherianity of $X$, implies that 
\[
\left\{\cJ \left(X, (c/p) \cdot \fa_{p} \right) \right\}_{p\in \N}
\]
has a unique maximal element that is called the \emph{$c$-th asymptotic multiplier ideal} 
and denoted by $\cJ(X,c\cdot \ab)$. Note that 
$\cJ(X,c\cdot \ab) = \cJ \left(X, (c/p) \cdot \fa_{p} \right)$ for all $p$ divisible enough. 

Asymptotic multiplier ideals also satisfy a subadditivity property.  
See~\cite[Theorem 11.2.3]{LazPAG} for the case when $X$ is smooth.
\begin{cor}\label{c:asub}
  Let $\ab$ be a graded sequence of ideals on $X$. 
  If $m\in\N^*$ and $c\in\Q_+^*$, then
  \[
  \left(\Jac_X\right)^{m-1} \cJ(X, cm\cdot \ab) \subset \cJ( X, c \cdot \ab)^m.
  \]
\end{cor}
Next we give a containment relation for the multiplier ideal associated to the graded sequence 
of valuation ideals. The result appears in~\cite{ELS} in the case when $v$ is divisorial. 
\begin{prop}\label{p:multval}
  If $v\in\Val_X$ is a valuation with $A(v)<\infty$, and $c\in\Q_+^*$, then
  \[ 
  \cJ(X,  c \cdot  \ab(v) ) \subset \fa_{c-A(v)} (v).
  \]
\end{prop}
\begin{proof}
  It is an immediate consequence of the valuative criterion for membership in the 
  multiplier ideal~\cite[Theorem1.2]{BdFFU} that 
  \[
  \cJ(X,  c \cdot  \ab(v) )
  \subset \fa_{c v(\ab(v)) - A(v)}(v).
  \]
  Since $v(\ab(v))=1$ (see~\cite[Lemma 3.5]{Blu16b}), the proof is complete. 
\end{proof}
%
%
%
%
\subsection{Multiplier ideals of linear series}
Given a linear series of $L$, we set
\begin{equation*}
  \cJ(X, c \cdot |V|):= \cJ(X, c \cdot \fb( |V|)),
\end{equation*}
where $\fb( |V|)$ is the base ideal of $V$. 
Similarly, if $V_\bullet$ is a graded linear series of $L$, we set
\begin{equation*}
  \cJ(X, c\cdot \lVert V_\bullet \|) := \cJ(X, c \cdot \fb_\bullet)
\end{equation*}
where $\fb_\bullet$ is the graded sequence of ideals defined by $\fb_m := \fb(|V_m|)$. 
We conclude
\begin{lem}\label{l:baseincl}
  Let $L$ be a line bundle on $X$. 
  \begin{enumerate}
  \item[(i)] 
    If $V$ is a linear series of $L$, then 
    $\fb(|V|) \subset \cJ( X, |V|)$.
  \item[(ii)] 
    If $V_\bullet$ is a graded linear series of $L$ and $m\in\N^*$, then 
    $\bs{V_m} \subset \cJ(X, m \cdot\|V_\bullet\| )$.
  \item[(iii)] 
    If $V_\bullet$ is a graded linear series of $L$ and $m\in\N^*$, $c\in\Q_+^*$, then 
    \begin{equation*}
      (\Jac_X)^{m-1}\otimes\cJ(X,cm\cdot\| V_\bullet\|)
      \subset \cJ(X, c\cdot\| V_\bullet\|)^m
    \end{equation*}
  \end{enumerate} 
\end{lem}
The following result is a consequence of Nadel Vanishing. 
\begin{thm}\label{t:nadel}
  Let $L$ be a big line bundle on $X$, and $V_\bullet$ a graded linear series of $L$. 
  \begin{itemize}
  \item[(i)] 
    Let $B$ be a line bundle on $X$ and $m \in \N^*$. 
    If  $B-K_X-mL$ is big and nef, then 
    \[
    H^i(X, B\otimes\cJ(X,m\cdot\| V_\bullet \| )) =0
    \]
    for all  $i \ge 1$.
  \item[(ii)]
    Let $B$ and $H$ be line bundles on $X$ and $m \in \N^*$. 
    If $H$ is ample and globally generated, and  $B-K_X- mL$ is big and nef, then 
    \[
      (B+jH)\otimes\cJ(X, m \cdot \| V_\bullet \| ) 
    \]
    is globally generated for every $j\ge n= \dim(X)$. 
  \end{itemize}
\end{thm}
\begin{proof}
  Statement (i) is~\cite[Theorem 11.2.12~(iii)]{LazPAG} in the case when $X$ is smooth. 
  When $X$ is klt, the statement is a consequence of~\cite[Theorem 9.4.17~(ii)]{LazPAG}. 
  
  Statement (ii) is a well known consequence of~(i) and Castelnuovo--Mumford regularity. 
  For a similar argument, see~\cite[Proposition 9.4.26]{LazPAG}. 
\end{proof}
\begin{cor}\label{c:uniformglobal}
  Let $L$ be an ample line bundle on $X$. 
  There exists a positive integer $a=a(L)$ such that if $V_\bullet$ 
  is a graded linear series of $L$,  then 
  \[
  (a+m)L \otimes \cJ(X, m \cdot  \| V_\bullet \| )
  \]
  is globally generated for all  $m\in\N^*$. 
  (Note that $a$ does not depend on $m$ or $V_\bullet$.)  
  Furthermore, we may choose $a$ so that $H^0(X,aL\otimes \Jac_X)$ is nonzero.  
\end{cor} 
\begin{proof}
  Pick $b,c\in\N^*$ such that $bL$ is globally generated and $cL-K_X$ is big and nef.
  We apply Theorem~\ref{t:nadel}~(ii) with $B= (c+m)L$ and $H=bL$. Thus
  \[
  (c+m+jb)L\otimes\cJ(X, m\cdot\| V_\bullet\|)
  \]
  is globally generated for all $m\in\N^*$ and $j\ge n$. 
  We can now set $a:=c+jb$, where $j\ge n$ is large enough so that
  $H^0(X,(c+jb)L\otimes \Jac_X)\ne0$.
\end{proof}
%
%
%
%
\subsection{Applications to filtrations defined by valuations}
Now let $L$ be an \emph{ample} line bundle on $X$ and 
fix a constant $a:=a(L)$  that satisfies the conclusion of Corollary~\ref{c:uniformglobal}. 
For the remainder of this section,  $a$ will always refer to this constant. 

Consider a valuation $v\in\Val_X^*$ with $A(v)<\infty$. 
We proceed to study the graded linear series $V_\bullet^t=V_{\bullet}^{\cF_v,t}$
of $L$ for $t \in \R_{+}$.
 
\begin{prop}\label{p:multvalinclusion}
  If $m\in\N^*$ and $t\in\Q_+^*$ satisfies $mt\ge A(v)$, then 
  \[
  \cJ(X, m \cdot \| V_\bullet^t \|) \subset \fa_{mt-A(v)} (v).
  \] 
\end{prop}
\begin{proof}
  Pick $p\in\N^*$ such that $pt\in\N^*$ and
  $\cJ(X, m \cdot \| V_\bullet^t \|) = \cJ(X,\frac{m}{p}\cdot\bs{V_p^t})$. 
  Then
  \[
  \cJ(X,\tfrac{m}{p}\cdot\bs{V_p^t}) \subset \cJ(X,\tfrac{m}{p}\cdot\fa_{pt}(v))
  \subset \cJ(X, mt \cdot \ab(v))  \subset \fa_{mt-A(v)}(v),
  \]
  where the first inclusion follows from the inclusion $\bs{V_p^t} \subset \fa_{pt}(v)$, 
  the second from the definition of the asymptotic multiplier ideal, and the 
  third from Proposition~\ref{p:multval}.
\end{proof}
\begin{prop}\label{p:basemultincl}
  If $m\in \N^*$ and $t\in \Q_+^*$ satisfies $mt\ge A(v)$, then
  \[
  \cJ(X, m \cdot \| V_\bullet ^t \|) \subset \bs{V_{m+a}^{t'}}
  \]
  where $t'= (mt-A(v))/(m+a)$.
\end{prop}
\begin{proof}
  By Proposition~\ref{p:multvalinclusion}, we have 
  \[
  H^0(X, (m+a)L\otimes\cJ(X,m\cdot\| V_\bullet ^t\|)) 
  \subset H^0(X, (m+a)L\otimes\fa_{mt-A(v)}(v)) 
  =V_{m+a}^{t'}.
  \]
  Since $(m+a)L\otimes \cJ(X, \| V_\bullet ^t \|)$ is globally generated
  by Corollary~\ref{c:uniformglobal}, the desired inclusion follows by taking
  base ideals.
\end{proof}
Using the previous proposition, we can now bound $\vol(V_{m,\bullet}^t)$ from below. 
\begin{prop}\label{p:multbase}
  If $m\in\N^*$ and $t\in\Q_+^*$ satisfies $mt\ge A(v)$, then 
  \[
  \vol(V_\bullet^t) \le m^{-n}\vol(V_{m+a,\bullet}^{t'}),
  \]
  where $t' = (mt-A(v))/(a+m)$.
  \end{prop}
\begin{proof}
  It suffices to show that $\dim V_{m\ell}^t \le \dim V^{t'}_{m+a,\ell}$ for all positive integers 
  $m$ and $\ell$. 
  Indeed, diving both sides by $(m\ell)^n/n!$ and letting $\ell\to\infty$ then 
  gives the desired inequality. 
  
  We now prove $\dim V_{m\ell}^t \le \dim V^{t'}_{m+a,\ell}$. First, by our assumption on $a$, 
  we may choose a nonzero section $ s \in H^0(X,aL\otimes \Jac_X)$. 
  Multiplication by $s^\ell$ gives an injective map 
  \[
  V_{\ell m}^t  \longrightarrow H^0(X, (a+m)\ell L\otimes (\Jac_X)^{\ell-1} \otimes \bs{V_{m\ell}^t}).
  \]
  Now, we have
  \begin{multline*}
    H^0(X, (a+m)\ell L\otimes (\Jac_X)^{\ell-1} \otimes \bs{V_{m\ell}^t})\\
    \subset H^0(X,(a+m)\ell L\otimes(\Jac_X)^{\ell-1}\otimes\cJ(X,m\ell\cdot\| V_\bullet^t\|) )\\
    \subset
    H^0(X,(a+m)\ell L\otimes \cJ(X, m \cdot   \| V_\bullet^t \|)^\ell)\\
    \subset 
    H^0(X,(a+m)\ell L\otimes(\fb(|V_{m+a}^{t'}|)^\ell)
    \subset  V^{t'}_{m+a,\ell},
  \end{multline*}
  where the first inclusion follows from Lemma~\ref{l:baseincl}, 
  the second from Corollary~\ref{c:asub}~(iii), 
  the third from Proposition~\ref{p:basemultincl},
  and the last one from the definition of $V^{t'}_{m+a,\bullet}$.
\end{proof}
As an application of the previous proposition, we give bounds 
on $T_m(v)$ and $\tilde{S}_m(v)$.
\begin{prop}\label{p:fujitaT}
  If $m\in\N^*$, then
  \begin{equation*}
    T(v) -\frac{aT(v) + A(v)}{m} \le T_m(v) \le T(v).
\end{equation*}
\end{prop}
\begin{proof}
  The second inequality is trivial, since $T(v) = \sup T_m(v)$. 
  To prove the first inequality, we may assume $m>a+\frac{A(v)}{T(v)}$.
  Pick $t\in\Q_+^*$ with $t<T(v)$ and $m>a+\frac{A(v)}{t}$.
  Since $V_\bullet^t$ is nontrivial (in fact, it contains an ample series),
  $\cJ(X,m\|V_\bullet^t\|)$ is nontrivial as well. 
  Apply Proposition~\ref{p:basemultincl}, with $m$ instead of $m-a$,
  so that $t'=t-m^{-1}(at+A)$. We get 
  \begin{equation*}
    \fb(|V_m^{t'}|)\supset\cJ(X,(m-a)\|V_\bullet^t\|)\ne0.
  \end{equation*}
  In particular, $V_m^{t'}\ne\emptyset$, which implies $t'\le T_m(v)$.
  Letting $t\to T(v)$ completes the proof.
\end{proof}
\begin{prop}\label{p:fujitaS}
  If $m\in\N^*$ and $m>a$, then
  \begin{equation}\label{e124}
    \left(\frac{m-a}{m}\right)^{n+1} \left( S(v)  - \frac{A(v)}{m-a}  \right) \le \tilde{S}_m(v)\le S(v).
  \end{equation}
\end{prop}. 
\begin{proof} 
  To prove the second inequality, note that for $t\in\R_+$ and $l\in\N^*$ we have 
  \begin{equation*}
    V_{m,\ell}^t 
    =H^0(X, m\ell L\otimes\overline{\bs{\cF_v^{mt}H^0(X,mL)}^\ell})
    \subset \cF_v^{m\ell t}H^0(X,m\ell L)
    =V_{m\ell}^t.
  \end{equation*}
  Thus $\vol(V^t_{m,\bullet})\le m^n\vol(V^t_\bullet)$ for $t\in\R_+$, 
  and integration yields $\tilde{S}_m(v)\le S(v)$.

  We now prove the first inequality.
  To this end, we use Proposition~\ref{p:multbase} 
  with $m$ replaced by $m-a$ to see that
  \begin{equation}\label{e123}
    \left(\frac{m-a}{m}\right)^n \vol(V_\bullet^t) 
    \le\frac{1}{m^n}\vol(V_{m,\bullet}^{t'})
  \end{equation}
  for all $t\in\Q_+^*$ with $(m-a)t\ge A(v)$, where $t' =t-m^{-1}(at+A(v))$. 
  By the continuity statement in Proposition~\ref{p:Brunn}, 
  the inequality in~\eqref{e123} must hold for all $t\in[m^{-1}A(v),T(v)]$,
  with at most two exceptions.
  We can therefore integrate with respect to $t$ from $t=A(v)/(m-a)$ to
  $t=(mT(v)+A(v))/(m-a)$, \ie~from $t'=0$ to $t'=T(v)$. This yields
  \begin{multline*}
    \tilde{S}_m(v)
    =\int_0^{T(v)}\frac{\vol(V_{m, \bullet}^{t'})}{m^n\vol(L)}\, dt'
    \ge\left(\frac{m-a}{m}\right)^{n+1}
    \int_{A(v)/(m-a)}^{(mT(v)+A(v))/(m-a)}\frac{\vol(V_\bullet^t)}{\vol(L)}\,dt\\
    =\left(\frac{m-a}{m}\right)^{n+1}
    \int_{A(v)/(m-a)}^{T(v)}\frac{\vol(V_\bullet^t)}{\vol(L)}\,dt\\
    =\left(\frac{m-a}{m}\right)^{n+1}
    \left(S(v)-\int_0^{A(v)/(m-a)}\frac{\vol(V_\bullet^t)}{\vol(L)}\, dt\right)\\
    \ge\left(\frac{m-a}{m}\right)^{n+1}
    \left(S(v)-\frac{A(v)}{m-a}\right),
  \end{multline*}
  where the second equality follows from a simple substitution
  and the last inequality follows since $\vol(V_\bullet^t) \le\vol(L)$ for all $t$.
  This completes the proof.
\end{proof}
\begin{proof}[Proof of Theorem~\ref{t:fujitaT}]
  Consider any $v\in\Val_X^*$ with $A(v)<\infty$.
  By Corollary~\ref{C105}, we have $T(v)\le A(v)/\a(L)$.
  Proposition~\ref{p:fujitaT} now yields 
  \begin{equation*}
    T(v)-T_m(v)\le\left(\frac{a}{\a(L)}+1\right)\frac{A(v)}{m}
  \end{equation*}
  for any $m\in\N^*$,  so the theorem holds with $C=1+a(L)/\a(L)$.
\end{proof}

\begin{proof}[Proof of Theorem~\ref{t:fujitaS}]
  Consider any $v\in\Val_X^*$ with $A(v)<\infty$.
  Proposition~\ref{p:fujitaS} gives
  \begin{multline*}
  0\leq  S(v)-\tilde{S}_m(v)
    \le S(v)-\left(\frac{m-a}{m}\right)^{n+1}\left(S(v)-\frac{A(v)}{m-a}\right)\\
    =\left(1-\left(\frac{m-a}{m}\right)^{n+1}\right)S(v)
    +\left(\frac{m-a}{m}\right)^n\frac{A(v)}{m}
    \le\frac{a(n+1)}{m}S(v)+\frac{A(v)}{m}
  \end{multline*}
  for $m>a$, where the last inequality uses that $1-t^{n+1}\le (n+1)(1-t)$ for $t\in[0,1]$.
  Since $S(v) \le A(v)/\d(L)$ by Theorem~\ref{t:delta}, we can take $C=1+(n+1)a(L)/\alpha(L)$.
\end{proof}
%
%
%
%
%
%
\section{Valuations computing the thresholds}\label{s:theoremE}
In this section we prove Theorem~E, on the existence of valuations computing the 
log canonical and stability thresholds.
We assume that $X$ is a normal projective klt variety and that $L$ is ample.
%
%
%
%
\subsection{Linear series in families}
We consider the following setup, which will arise in~\S\ref{ss:limitpoints}. 
Fix $m\in\N^*$ and a family of subspaces of $H^0(X,mL)$ parameterized by a variety $Z$. Said family is given by a submodule
\[
\cW \subset \cV:= H^0(X,mL) \otimes_\C \cO_Z.\]
For $z \in Z$ closed, we write $W_z$ for the linear series of $mL$ defined by
\[
W_z := \im\left( \cW\vert_{k(z)}  \to \cV\vert_{k(z)} \simeq  H^0(X,mL) \right).\]
Note that $\cW$ gives rise to an ideal $\cB \subset \cO_{X\times Z}$ such that \[\cB \cdot \cO_{X\times \{z\}} =
\bs{W_z}.\]
Indeed, $\cB$ is the image of the map 
\[
  p_2^\ast \cW \otimes  p_1^\ast (-mL) \to \cO_{X\times Z},
\]
where $p_1$ and $p_2$ denote the projection maps associated to $X\times Z$. 

We need a few results on the behavior of invariants of linear series in families. 
\begin{prop}\label{p:constantlct}
  There exists a nonempty open set $U\subset Z$ such that $\lct( \bs{W_z})$ is constant 
  for all closed points $z\in U$.
\end{prop}
\begin{proof}
  Since $\lct(\bs{W_z})= \lct(\cB \cdot \cO_{X\times \{z\}})$, 
  the proposition follows from the well known fact that the log canonical threshold of a family 
  of ideals is constant on a nonempty open set; see~\eg~\cite[Proposition~A.2]{Blu16b}.
\end{proof}
\begin{prop}\label{p:lsclct}
  If $Z$ is a smooth curve and $z_0 \in Z$ a closed point, then there exists an 
  open neighborhood $U$ of $z_0$ in $Z$ such that 
  $\lct( \bs{W_{z_0}}) \le \lct(\bs{W_z})$ for all $z\in U$. 
\end{prop}
\begin{proof}
  As in the proof of the previous proposition, we note that 
  $\lct( \bs{W_z})= \lct(\cB \cdot \cO_{X\times \{z\}})$ for $z\in Z$ closed. 
  Thus, the proposition is a consequence of the lower semicontinuity of the 
  log canonical threshold. See~\cite[Proposition~A.3]{Blu16b}.
\end{proof}
Denote by $W_{z,\bullet}$ the graded linear series of $mL$ defined by 
\[
  W_{z,\ell} := H^0(X, m \ell L \otimes \overline{\bs{W_z}^\ell}).
\]
\begin{prop}\label{p:constantv}
  There exists a nonempty open set $U\subset Z$
  such that $\vol({W_{z,\bullet}})$ is constant for all closed points $z\in U$. 
\end{prop}
\begin{proof}
  The idea is to express $\vol(W_{z,\bullet})$ as an intersection number. 
  Fix a proper birational morphism $\pi\colon Y \to X\times Z$ 
  such that $Y$ is smooth and $\cB \cdot \cO_Y=  \cO_Y(-F)$ for some effective Cartier divisor 
  on $Y$. 
  For each $z\in Z$, we restrict $\pi$ to get a map $\pi_z\colon Y_z \to X\times \{z\} \simeq X$. 
  By generic smoothness, there exists a nonempty open set $U\subset Z$ such that 
  $Y_z$ is smooth for all $z \in U$. 
  For $z\in U$, we then have 
  \[
    \vol({W_{z,\bullet}})=
    ((p_1^\ast mL-F )\vert_{Y_z}^n).
  \]
  After shrinking $U$, we may assume $p_1^\ast mL-F$ is flat over $U$.
  Then $((p_1^\ast mL-F)\vert_{Y_z})^n)$ is constant on $U$, which concludes the proof.
\end{proof}
\begin{prop}\label{p:inclusion}
  Let $\cW$ and $\cG$ be two submodule of $\cV$ and for $z\in Z$, let $W_z$ and $G_z$ 
  denote the corresponding subspaces of $V$. If the function $z \mapsto \dim W_z$ 
  is locally constant on $Z$,  then the set $\{z \in Z\,\vert\, G_z \subset W_z\}$ is closed. 
\end{prop}
\begin{proof}
  We may assume $Z$ is affine and $\dim(W_z)=:r$ is constant on $Z$.
  Choose a basis for the free $\cO(Z)$-module $\cV(Z)$ as well as  
  generators for $\cW(Z)$ and $\cG(Z)$.
  Consider the matrix with entries in $\cO(Z)$, whose rows are given by the 
  generators of $\cW(Z)$, followed by the generators of $\cG(Z)$,
  all expressed in the chosen basis of $\cO(Z)$.
  By our assumption on $\cW$, the rank of this matrix is at least $r$ for all $z\in Z$.
  Further, since $G_z \subset W_z$ if and only if $\dim (G_z + W_z )= \dim(W_z)$, 
  the set $\{z \in Z\,\vert\, G_z \subset W_z \}$ is precisely the locus 
  where this matrix has rank equal to $r$, and is hence closed.
\end{proof}
%
%
%
%
\subsection{Parameterizing filtrations}\label{ss:parameterspace}
We now  construct a space that parameterizes filtrations of 
$R(X,L)$.
\footnote{See~\cite{Cod18} for a related, but different, construction that parameterizes limits of test configurations.}
To have a manageable parameter space, 
we restrict ourselves to $\N$-filtrations $\cF$ of $R$ satisfying $T(\cF)\le1$. 
Such a filtration $\cF$  is given by the choice of a flag
\begin{equation}\label{e121}
  \cF^m R_m  \subset \cF^{m-1} R_m 
  \subset \cdots \subset \cF^1 R_m \subset \cF^0 R_m = R_m
\end{equation}
for each $m\in\N^*$ such that 
\begin{equation}\label{e:multiplicative}
  \cF^{p_1} R_{m_1} \cdot  \cF^{p_2} R_{m_2} \subset \cF^{p_1+p_2}R_{m_1+m_2}
\end{equation}
for all integers $0 \le p_1 \le m_1$ and $0 \le p_2 \le m_2$. 

Let $Fl_m$ denote the flag variety parameterizing flags of $R_m$ 
of the form~\eqref{e121}.
In general, $Fl_m$ may have several connected components. 
On each component, the signature of the flag 
(that is, the sequence of dimensions of the elements of the flag) is constant. 

For each natural number $d$, we set 
\[ H_d := Fl_0 \times Fl_1 \times \cdots \times Fl_d\]
and, for $c \ge d$, let $\pi_{c,d}:H_{c} \to H_{d}$ 
denote the natural projection map. Note that a closed point $z\in H_d$ gives a 
collection of subspaces 
\[
  \left(\cF_z^m R_m \subset \cF_z^{m-1} R_m \subset \cdots 
    \subset \cF_z^1 R_m \subset \cF_z^0 R_m = R_m \right)_{0 \le m \le d }.
\]
Furthermore, this correspondence is given by a universal flag on $H_d$. 
This means that for each $m \le d$ on $H_d$ there is a flag
\[
  \cF^m \cR_m \subset \cF^{m-1} \cR_{m} \subset \cdots 
  \subset \cF^{1} \cR_m \subset \cF^0 \cR_m= \cR_m,
\]
where $\cR_m := H^0(X,mL)\otimes_\C \cO_{H_d}$. 
For $z\in H_d$, we have
\[
  \cF_z^p R_m := \im \left( {\cF^p \cR_m}\vert_{k(z)} 
    \longrightarrow {\cR_m}\vert_{k(z)} \simeq R_m \right) 
\]
for $0 \le p\le m$, where $k(z)$ denotes the residue field at $z$.

Since we are interested in filtrations of $R(X,L)$, consider the subset 
\[
  J_d := \{z \in H_d \mid \cF_z\ 
  \text{satisfies~\eqref{e:multiplicative} for all $0 \le p_i \le m_i \le d$}\}.
\] 

\begin{lem}
  The subset $J_d \subset H_d$ is closed. 
\end{lem}
\begin{proof}
  We consider $\cF_z^{p_1} R_{m_1} \cdot \cF_{z}^{p_2} R_{m_2}$, 
  where $z\in H_d$, $m_1+m_2 \le d$, and $0\le p_i \le m_i$ for $i=1,2$. 
  We will realize this subspace as coming from a submodule of $\cR_{m_1+m_2}$. 
  Note that the natural map
  \[
    H^0(X, m_1L)\otimes_{k} H^0(X,m_2 L) \longrightarrow H^0(X, (m_1+m_2)L) 
  \]
  induces a map $\cR_{m_1} \otimes \cR_{m_2}  \to \cR_{m_1+m_2}$.
  We define 
  \[
    \cF^{p_1} \cR_{m_1} \cdot \cF^{p_2} \cR_{m_2} := 
    \im\left(
      \cF^{p_1} \cR_{m_1} \otimes \cF^{p_2} \cR_{m_2}
      \to \cR_{m_1+m_2} \right).
  \]
Since
  \[
    \cF_z^{p_1} R_{m_1} \cdot \cF_z^{p_2} R_{m_2}
    =\im\left( 
      (\cF^{p_1} \cR_{m_1}\otimes \cF^{p_2} \cR_{m_2})\vert_{k(z)} 
      \longrightarrow \cR_{m_1+m_2} \vert_{k(z)} \simeq R_{m_1+m_2} 
    \right),
  \] 
  the desired statement is a consequence of Proposition~\ref{p:inclusion}.
\end{proof}

Let $J_d(\C)$ denote the set of closed points of $J_d$, and set
$J:= \varprojlim J_d(\C)$, with respect to the inverse system
induced by the maps $\pi_{c,d}$. Write $\pi_{d}$ for the natural map $J  \to J_{d}(\C)$
By the previous discussion, there is a bijection between the elements of $J$ 
and $\N$-filtrations $\cF$ of $R(X,L)$ satisfying $T(\cF)\le 1$.

\medskip
The following technical lemma will be useful for us in the next section. 
Its proof relies on the fact that every descending sequence of nonempty constructible 
subsets of a variety over an uncountable field has nonempty intersection. 
\begin{lem}\label{l:nonempty}
  For each $d\in\N$, let $W_d\subset J_d$ be a nonempty constructible subset,
  and assume $W_{d+1} \subset \pi_{d+1,d}^{-1}(W_d)$ for all $d$.
  Then there exists $z\in J$ such that $\pi_d(z)\in W_d(\C)$ for all~$d$.
\end{lem}
\begin{proof}
  Finding such a point $z$ is equivalent to finding a point $z_d\in W_d(\C)$
  for each $d$, such that $\pi_{d+1,d}(z_{d+1})=z_d$ for all $d$.
  We proceed to construct such a sequence $(z_d)_d$ inductively. 

  We first look to find a good candidate for $z_1$. By assumption, 
  \[
    W_1 \supset \pi_{2,1}(W_2) \supset \pi_{3,1}(W_3) \supset \cdots 
  \]
  is a descending sequence of nonempty sets. Note that $W_1$ is constructible, 
  and so are $\pi_{d,1}(W_d)$ for all $d$ by Chevalley's Theorem. Thus, 
  \[
    W_1 \cap  \pi_{2,1}(W_2) \cap \pi_{3,1}(W_3) \cap \cdots 
  \]
  is nonempty, and we may choose a closed point $z_1$ in this set. 
  
  Next, we look at
  \[
    W_2 \cap \pi_{2,1}^{-1}(z_1) \supset \pi_{3,2}(W_3) \cap \pi_{2,1}^{-1}(z_1) 
    \supset \pi_{4,2}(W_4)\cap \pi_{2,1}^{-1}(z_1) \supset \cdots
  \]
  and note that for $d\ge 2$ the set $\pi_{d,2}(W_d)\cap \pi_{2,1}^{-1}(z_1)$ 
  is nonempty by our choice of $z_1$. Thus
  \[
    \pi_{2,1}^{-1}(z_1) \cap W_2 \cap \pi_{3,2}(W_3) \cap  \pi_{4,2}(W_4)\cap  \cdots
  \]
  is nonempty, and we may choose a closed point $z_2$ lying in the set. 
  Continuing in this manner, we construct a desired sequence. 
\end{proof}
%
%
%
%
\subsection{Finding limit filtrations}\label{ss:limitpoints}
The following proposition, crucial to Theorem~E,
is a global analogue of~\cite[Proposition 5.2]{Blu16b}. 
The proofs of both results use extensions of the  ``generic limit'' 
construction developed in~\cite{Kol08,dFM09,dFEM10,dFEM11}.  
\begin{prop}\label{p:genlimit}
  Let  $(\cF_{i})_{i \in \N}$ be a sequence of $\N$-filtrations of $R(X,L)$  
  with $T(\cF_i) \le 1$ for all $i$.
  Furthermore, fix $A,S,T \in \R_{+}$ such that
  \begin{enumerate}
  \item 
    $A\ge \limsup\limits_{i \to \infty}  \lct \left( \fb_\bullet(\cF_i) \right)$, 
  \item 
    $S \le \liminf\limits_{m \to \infty}\liminf\limits_{i \to \infty} \tilde{S}_m(\cF_i)$, and
  \item 
    $T\le  \liminf \limits_{m \to \infty} \liminf\limits_{i \to \infty}  T_m(\cF_i)$. 
  \end{enumerate}
  Then there exists a filtration $\cF$ of $R(X,L)$ such that 
  \[
    \lct\left( \fb_{\bullet}(\cF) \right) \le A,
    \quad
    S(\cF) \ge S,
    \quad
    \text{and }
    \quad
    T \le T(\cF) \le 1.
  \]
\end{prop}
\begin{proof}
  We use the parameter space $J$ from~\S\ref{ss:parameterspace},
  parametrizing $\N$-filtrations of $R(X,L)$ with  $T\le 1$. 
  Each filtration $\cF_i$ corresponds to an element $z_i\in J$,
  and $\pi_m(z_i)$ correspond to the filtration $\cF_i$ restricted to $\oplus_{d=0}^m R_d$. 

  \medskip
  \noindent {\bf Claim 1:} We may choose infinite subsets 
  \[
    \N\supset I_0 \supset I_2 \supset I_3 \supset \cdots 
  \]
  such that for each $m$, the closed set   
  \[
    Z_m:= \overline{\{\pi_m(z_i) \, \vert \, i \in I_m\}} \subset J_m
  \]
  satisfies the property
  \begin{center}
    $(\dagger)$ If $Y\subsetneq Z_m$ is a closed set, 
    there are only finitely many $i \in I_m$ such that $\pi_m(z_i) \in Y$.
  \end{center} 
  Note that, in particular, each $Z_m$ is irreducible.

  \smallskip
  Indeed, we can construct the sequence $(I_m)_0^\infty$ inductively. Set $I_0 = \N$. 
  Since $J_0 = Fl_0 \simeq\Spec(\C)$, $(\dagger)$ is trivially satisfied for $m=0$. 
  Having chosen $I_m$, pick $I_{m+1}\subset I_m$ such that $(\dagger)$ is satisfied 
  for $Z_{m+1}$; this is possible since $J_m$ is Noetherian.

  \medskip
  \noindent {\bf Claim 2:} 
  For each $m\in \N$, there exist a nonempty open set $U_m\subset Z_m$ and constants 
  $a_{p,m}$, $1\le p\le m$, $s_m$, and $t_m$ such that if $z\in U_m$, the filtration $\cF_z$  satisfies
  \begin{enumerate}
  \item[(1)] 
    $p\cdot\lct\left( \fb_{p,m}(\cF_z) \right) =a_{p,m}$ for $1\le p\le m$;
  \item[(2)]
    $\tilde{S}_{m}(\cF_z) = s_m$;
  \item[(3)]
    $T_m(\cF_z) =t_m$.
  \end{enumerate}  
  Furthermore, $a_{p,m}\le A$ for all $1 \leq p \leq m$, 
  $\liminf \limits_{m \to \infty} s_m\ge S$, 
  and $\liminf \limits_{m \to \infty} t_m\ge T$.
  
  \smallskip
  To see this, note that there is a nonempty open set $U_m \subset Z_m$ on 
  which the left-hand sides of~(1)--(3) are constant. 
  For (1) and (2), this is a consequence of Propositions~\ref{p:constantlct} 
  and~\ref{p:constantv}. For (3), it follows from 
  $\dim\cF_z^pR_m$ being constant on the connected components of $J_m$. 

  Now, we let 
  \[
    I_m^\circ := \{i \in I_m \, \vert\,  \pi_m(z_i) \in U_m \}.
  \]
  By~($\dagger$), the set $I_m \setminus I_m^\circ$ is finite; hence, $I_m^\circ$ is infinite. 
  Since
  \[
    a_{p,m} = p \cdot  \lct(\fb_{p,m}(\cF_i)), \quad
    s_m = \tilde{S}_m(\cF_i), \quad\text{and} \quad
    t_m= T_m(\cF_i)
  \]
  for all $i\in I_m^\circ$ and $1\le p\le m$, we see that
  \begin{enumerate}
  \item 
    $a_{p,m} \le  \limsup \limits_{i \to \infty}  p \cdot \lct( \fb_{p,m}(\cF_i) ) \le   \limsup \limits_{i \to \infty}  p \cdot \lct(\fb_{p}(\cF_i))$,
  \item 
    $s_m \ge \liminf \limits_{i\to \infty} \tilde{S}_m(\cF_i) $, and
  \item  $t_m \ge \liminf \limits_{i \to \infty} T_m(\cF_i)$.
  \end{enumerate}
  The remainder of Claim 2 follows from these three inequalities.

  \medskip
  \noindent {\bf Claim 3:}
  There exists a point $z\in J$ such that $\pi_m(z)\in U_m$ for all  $m \in \N$.

  \smallskip
  Granted this claim, the filtration $\cF=\cF_z$ associated to $z\in J$
  satisfies the conclusion of our proposition. 
  Indeed, this is a consequence of Claim 2 and the fact that for any 
  linearly bounded filtration $\cF$, we have 
  \begin{enumerate}
  \item 
    $\lct(\fb_{\bullet}(\cF)) = \lim_{p\to\infty}\sup_{m\ge p} p\cdot \lct( \fb_{p,m} (\cF) )$;
  \item 
    $S(\cF) = \lim_{m \to \infty} \tilde{S}_m(\cF)$;
  \item  
    $T(\cF)=\lim_{m\to\infty} T_m(\cF)$. 
  \end{enumerate}
  
  \smallskip
  We are left to prove Claim 3. To this end we apply Lemma~\ref{l:nonempty}. 
  For $d\in\N$, set
  \[
    W_d := U_d \cap \pi_{d,d-1}^{-1} U_{d-1} \cap  \pi_{d,d-2}^{-1}(U_{d-2}) 
    \cap \cdots \cap \pi_{d,0}^{-1} (U_0).
  \]
  Clearly $W_d\subset J_d$ is constructible and 
  $W_{d+1} \subset \pi_{d+1,d}^{-1}(W_d)$. 
  We are left to check that each $W_d$ is nonempty. But 
  \[
    \pi_{d}(z_i) \in W_d\ \text{for all } i \in I_d^\circ \cap I_{d-1}^\circ \cdots \cap I_0^\circ,
  \]
  and the latter index set is nonempty, since it can be written as 
  $I_d\setminus\bigcup_{j=0}^d(I_j\setminus I_j^\circ)$,
  where $I_d$ is infinite and each $I_j \setminus I_j^\circ$ is finite.

  Applying Lemma~\ref{l:nonempty} to the $W_d$ yields a point $z\in J$ 
  such that $\pi_d(z)\in W_d\subset U_d$ for all $d\in\N$. 
  This completes the proof of the claim, as well as the proof of the proposition. 
\end{proof}
%
%
%
%
\subsection{Proof of Theorem E} 
 We begin by proving the following proposition. 
\begin{prop}\label{p:limitval}
  Let $(v_i)_{i\in\N}$ be a sequence of valuations in $\Val_X^*$ such that $T(v_i)=1$ 
  and the limits 
  $A:=\lim_{i \to \infty} A(v_i)$ and $S:=\lim_{i \to \infty} S(v_i)$ both exist and are finite.
  Then there exists a valuation $v^\ast$ on $X$ such that 
  \[ 
    A(v^\ast) \le A,\quad 
    S(v^\ast) \ge S 
    \quad\text{and}\quad
    T(v^\ast) \ge 1.
  \]
\end{prop}
This will follow from Proposition~\ref{p:genlimit} and the following lemma. 
\begin{lem}\label{l:approximations}
  Keeping the notation and hypotheses of Proposition~\ref{p:limitval}, 
  let $\cF_i:=\cF_{v_i, \N}$ denote the $\N$-filtration induced by $\cF_{v_i}$ 
  as in~\S\ref{ss:Nfiltrations}. Then we have
  \begin{enumerate}
  \item $\limsup\limits_{i \to \infty}  \lct \left( \fb_\bullet(\cF_i) \right)\le A$, 
  \item $\lim\limits_{m \to \infty}\liminf\limits_{i \to \infty}\tilde{S}_m(\cF_i)
    =\lim\limits_{m \to \infty}\limsup\limits_{i \to \infty}\tilde{S}_m(\cF_i)=S$, and
  \item $\lim\limits_{m \to \infty}\liminf\limits_{i \to \infty}T_m(\cF_i)
    =\lim\limits_{m \to \infty}\limsup\limits_{i \to \infty}T_m(\cF_i)=1$. 
  \end{enumerate}
\end{lem}
\begin{proof}
  We first show that (1) holds. Note that $\fb_{p}(\cF_i)= \fb_{p}(\cF_{v_i})$ for all $p \in \N$. 
  Indeed, this follows from the fact that $\cF_i^p R_m = \cF_{v_i}^p R_m$ for all $m, p \in \N$. 
  Thus, 
  \[
    \lct( \fb_{\bullet}(\cF_i)) = \lct( \fb_{\bullet}(\cF_{v_i})) = \lct( \ab(v_i))\le A(v_i),
  \] 
  where the second equality follows from Lemma~\ref{l:basevalideals} and the last inequality 
  is Lemma~\ref{l:lctval}. 
  
  We now show (2) and (3) hold. To this end, we first claim that
  \begin{equation}\label{e:approximations}
    0\le T_m(v_i)- T_m(\cF_i)\le\frac1m
    \quad\text{and}\quad
    0\le \tilde{S}_m(v_i)- \tilde{S}_m(\cF_i)\le \frac{1}{m}.
  \end{equation}
  Indeed, the estimates for $T_m$ follow from Proposition~\ref{p:truncations}.
  As for the estimates for $\tilde{S}_m$, note that 
  $\tilde{S}_m(v_i)=\int_0^1f_{i,m}(t)\,dt$, where $f_{i,m}(t)=\vol(V_{m, \bullet}^{\cF_{v_i} ,t})$,
  whereas $\tilde{S}_m(\cF_i)$ is a right Riemann sum approximation of this integral,
  obtained by subdividing $[0,1]$ into $m$ subintervals of equal length.
  Thus the estimate for $\tilde{S}_m$ in~\eqref{e:approximations}
  follows, since the functions $f_{i,m}(t)$ 
  are decreasing, with $f_{i,m}(0)=1$ and $f_{i,m}(1)\ge0$.
  
  By the uniform Fujita approximation results in Theorems~\ref{t:fujitaT} and~\ref{t:fujitaS}, 
  we have 
  \begin{equation*}
    \lim_{m\to\infty}\sup_i|T_m(v_i)-T(v_i)|
    =\lim_{m\to\infty}\sup_i|\tilde{S}_m(v_i)-\tilde{S}(v_i)|
    =0.
  \end{equation*}
  Together with~\eqref{e:approximations}, this yields~(2) and (3), and hence 
  completes the proof.
\end{proof}

\begin{proof}[Proof of Proposition~\ref{p:limitval}]
  For $i\ge 1$, consider the $\N$-filtrations $\cF_i:=\cF_{v_i, \N}$ 
  associated to $v_i$.
  By Lemma~\ref{l:approximations}, the assumptions of Proposition~\ref{p:genlimit}
  are satisfied with $T=1$. Hence we may find a filtration $\cF$ such that 
  \[
    \lct( \fb_\bullet(\cF)) \le A,
    \quad
    S(\cF) \ge S
    \quad\text{and}\quad
    T(\cF) = 1.
  \]
  Using~\cite{jonmus}, we may choose a valuation $v^\ast\in\Val_X^*$ 
  computing $\lct(\fb_\bullet(\cF))$.
  After rescaling, we may assume ${v^\ast(\fb_\bullet(\cF)) =1}$. 
 Therefore, 
 \[
   A(v^\ast)= 
   \frac{A(v^\ast)}{v^\ast( \fb_\bullet(\cF))} = \lct(\fb_\bullet(\cF) \le A.
 \]
 By Corollary~\ref{c:vFinequality}, $S(v^\ast) \ge S(\cF)\ge S$ 
 and $T(v^\ast) \ge T(\cF) = 1$. This completes the proof.
\end{proof}
\begin{proof}[Proof of Theorem E]
  We first find a valuation computing $\a(L)$. 
  Choose a sequence  $(v_i)_i$ in $\Val_X^*$ such that 
  \[
    \lim_{i \to \infty} \frac{A(v_i)}{T(v_i)} 
    =
    \inf_{v} \frac{A(v)}{T(v)} = \a(L). \]
  After rescaling, we may assume $T(v_i)=1$ for all $i$. 
  Hence, the limit $A:= \lim_{i \to \infty} A(v_i)$ exists and equals $\a(L)$. 
  Further, by~\eqref{e122},
  the sequence $(S(v_i))_i$ is bounded from above and below away 
  from zero, so after passing to a subsequence we may assume the limit 
  $S:= \lim_{i \to \infty} S(v_i)$ exists, and is finite and positive. 

  By Proposition~\ref{p:limitval}, there exists $v^\ast\in\Val_X^*$ with $A(v^\ast) \le A$
  and $T(v^\ast)\ge 1$. Therefore, 
  \[
    \frac{A(v^\ast)}{T(v^\ast)} \le A = \a(L). 
  \]
  Since $\a(L)= \inf_{v} A(v)/T(v)$, $v^*$ computes $\a(L)$.

  \smallskip
  The argument for $\d(L)$ is almost identical.
  Pick a sequence $(v_i)_i$ in $\Val_X^*$ such that 
  \[
    \lim_{i \to \infty} \frac{A(v_i)}{S(v_i)} 
    =
    \inf_{v} \frac{A(v)}{S(v)} = \d(L).
  \]
  Again, we rescale our valuations so that $T(v_i) =1$ for all $i \in \N$. 
  As above, we may assume that the limit 
  $S:= \lim_{i \to \infty} S(v_i)$ exists, and is finite and positive. 
  Therefore, $A:= \lim_{i \to \infty} A(v_i)$ also exists and $A/S = \d(L)$. 

  We apply Proposition~\ref{p:limitval} to find a valuation $v^\ast$ such that $A(v^\ast) \le A$ and 
  $S(v^\ast) \le S$. As argued for $\a(L)$, we see that $v^*$ computes $\d(L)$.
\end{proof}
%
%
%
%
%
%
\section{The toric case}\label{S113}
In this section we will freely use notation and results found in~\cite{FultonToric}.
Fix a toric variety $X=X(\Delta)$ given by a  fan $\Delta$ in a lattice $N\simeq \Z^n$.
We assume that $X$ is proper and $K_X$ is $\Q$-Cartier. 
Set $N_\R := N \otimes_{\Z} \R$.

We write $M = \Hom(N, \Z)$, $M_\Q=M\otimes_\Z\Q$, and $M_\R = M \otimes_\Z \R$ for the corresponding dual lattice and vector spaces. 
The open torus of $X$ is denoted by $T\subset X$.  
Let $v_1, \ldots, v_d$ denote the primitive generators of the one-dimensional 
cones in $\Delta$ and let $D_1, \ldots, D_d$ be the corresponding torus invariant 
divisors on $X$. 
 
We fix an ample line bundle of the form $L = \cO_X(D)$, 
where $D=b_1 D_1+ \cdots + b_d D_d$ is
a Cartier divisor on $X$. Associated to $D$ is the convex polytope 
\[
  P=P_D= \{u \in M_\R \, \vert \, \langle u, v_i \rangle \ge -b_i \text{ for all } 1 \le i \le d \}.
\]
We write $\ver P$ for the set of vertices in $P$.

Recall that there is a correspondence between points in $P \cap M_\Q$ and 
effective torus invariant $\Q$-divisors $\Q$-linearly equivalent to $D$, under which 
$u \in P \cap M_\Q$ corresponds to 
\[
  D_u 
  := D+ \sum_{i=1}^{d} \langle u , v_i \rangle D_i
  := \sum_{i=1}^{d} (\langle u , v_i \rangle+b_i)D_i.
\]
Note that if $m\in\N^*$ is chosen so that $mu\in N$, then $D_u=D+m^{-1}\div(\chi^{mu})$. 

Let $\p=\p_D\colon N_\R \to \R$ be the concave function
that is linear on the cones of $\Delta$ and satisfies $\psi(v_i)=-b_i$ for $1 \le i \le d$. 
On a given cone $\sigma \in \Delta$, the linear function is
given by $\psi(v) = -\langle b(\sigma), v \rangle$, where $b(\sigma)\in M$
is such that $\chi^{b(\sigma)}$ is a local equation for $D$ on $U_\sigma \subset X$.  
We have $\psi(v)=\inf_{u\in P}\langle u,v\rangle=\min_{u\in\ver P}\langle u,v\rangle$
for all $v\in N_\R$.
%
%
%
%
\subsection{Toric valuations}
Given $v\in N_\R$, let $\sigma$ be the unique cone in $\Delta$ containing $v$ in its interior.
The map 
\[
  \C[\sigma^\vee \cap M ] = \bigoplus_{u \in \sigma^\vee \cap M}\C \cdot \chi^u \to \R_+
\]
defined by 
\begin{equation}\label{e126}
  \sum_{u \in \sigma^\vee \cap M} c_u \chi^u  \mapsto 
  \min \{ \langle u, v \rangle \, \vert \, c_u \neq 0 \}
\end{equation}
gives rise to a valuation on $X$ that we slightly abusively also denote by $v$. 
Its center on $X$ is the generic point of $V(\sigma)$. This induces in embedding
$N_\R\hookrightarrow\Val_X$, and we shall simply view $N_\R$ as a subset of $\Val_X$.
The valuations in $N_\R$ are called \emph{toric valuations}. 
The valuation associated to the point $v_i\in N_\R$ is $\ord_{D_i}$ for $1\le i\le d$,
and the valuation associated to $0\in N_\R$ is the trivial valuation on $X$.
\begin{lem}\label{l:toriceval}
  If $u \in P \cap M_\Q$ and $v \in N_\R$, then $v(D_u) = \langle u, v \rangle - \p(v)$. 
\end{lem}
\begin{proof}
  Pick $m\in\N^*$ such that $mu\in M$.  
  Since $D_u=D+m^{-1}\div (\chi^{mu})$, we have
  \[
    v(D_u)
    = v(D)+m^{-1}v(\chi^{mu}) 
    =v(D)+\langle u, v \rangle,
  \]
  and we are left to show $v(D) =-\p(v)$. 
  Let $\sigma \in \Delta$ be the unique cone containing $v$ in its interior.
  Since $\chi^{b(\sigma)}$ is a local equation for $D$ on $U_\sigma$, we see 
  \[
    v(D) = v( \chi^{b(\sigma)}) = \langle b(\sigma), v \rangle 
    = -\p(v),
  \]
  which completes the proof. 
\end{proof}
%
%
%
%
\subsection{Log canonical thresholds}
The following result is probably well known, but we include a proof 
for lack of a suitable reference.
\begin{prop}\label{p:Atoric}
  The restriction of the log discrepancy function $A=A_X$ to $N_\R\subset\Val_X$ is
  the unique function that is linear on the cones in $\Delta$ and satisfies $A(v_i)=1$ 
  for $1 \le i \le d$.
\end{prop}
\begin{proof}
  Consider any cone $\sigma\in\D$. Let $v_i\in N$, $1\le i\le r$, be the generators
  of the 1-dimensional cones contained in $\sigma$, and $D_i$, $1\le i\le r$
  the associated divisors on $X$. 
  Since $K_X$ is $\Q$-Cartier, there exists $b(\sigma)\in M_\Q$ such that 
  $\langle b(\sigma),v_i\rangle=-1$ for $1\le i\le r$. 
  Thus $K_X=-\sum_{i=1}^rD_i=\div_X(\chi^{b(\sigma)})$ on $U(\sigma)$.

  Pick any refinement $\D'$ of $\D$ such that $X':= X(\D')$ is smooth.
  Consider a cone $\sigma'\in\D'$ with $\sigma'\subset\sigma$.
  Let $v'_j\in N$ and $D'_j$, $1\le j\le s$, be the analogues of $v_i$ and $D_i$.
  Now 
  \begin{equation*}
    K_{X'/X}
    =K_{X'}-\div_{X'}(\chi^{b(\sigma)})
    =-\sum_{j=1}^sD'_j-\div_{X'}(\chi^{b(\sigma)})
  \end{equation*}
  on $U(\sigma')$.
  By the definition of the log discrepancy, this implies
  \begin{equation*}
    A_X(v'_j)
    =1+v'_j(K_{X'/X})
    =1-1-\langle b(\sigma),v'_j\rangle
    =-\langle b(\sigma),v'_j\rangle.
  \end{equation*}
  Since $\D'$ was an arbitrary regular refinement of $\D$, this implies that 
  the restriction of $A_X$ to $\sigma\subset N_\R\subset\Val_X$ is given
  by the linear function $b(\sigma)\in M_\Q$. This concludes the proof.
\end{proof}
The next proposition follows from~\cite[Proposition 8.1]{jonmus}. 
We say that ideal $\fa $ on $X$ is \emph{$T$-invariant} if it is invariant with 
respect to the torus action on $X$. Equivalently, for each $\sigma \in \Delta$, 
the ideal $\fa(U_\sigma) \subset k[\sigma^\vee \cap M]$ is generated by monomials.
\begin{prop}\label{p:Tvalcompute}
  If $\ab$ is a nontrivial graded sequence of $T$-invariant ideals on $X$, 
  then there exists a nontrivial toric valuation computing $\lct(\ab)$.
  Further, any valuation that computes $\lct(\ab)$ is toric.
\end{prop}
\begin{proof}
  Pick a refinement $\Delta'$ of $\Delta$ such that $X':= X(\Delta')$ is smooth. 
  This induces a proper birational morphism $X'\to X$. 
  Let $D'$ be the sum of the torus invariant divisors on $X'$. 
  
  By~\cite{jonmus}, there exists a valuation $w\in \Val_X$ computing $\lct(\ab)$. 
  We now follow~\cite[\S8]{jonmus}. Let $r_{X',D'}:\Val_X \to \QM(X', D')=N_\R$ 
  denote the retraction map defined in~\textit{loc.\,cit}, and set $v:=r_{X',D'}(w)\in N_\R$.
  Then $v(\ab)=w(\ab)>0$. In particular, $v$ is nontrivial.
  Further, $A_{X'}(v) \le A_{X'}(w)$, with equality iff $w=v\in N_\R$.
  Now recall that $A_X(v)=A_{X'}(v)+v(K_{X'/X})$ and $A_X(w)=A_{X'}(w)+w(K_{X'/X})$.
  Since $K_{X'/X}$ is $T$-invariant, we have $v(K_{X'/X})=w(K_{X'/X})$.
  This implies $A_X(v) \le A_X(w)$, with equality iff $w=v$.
  Thus $\lct(\ab)\le A_X(v)/v(\ab) \le A_X(w)/w(\ab)=\lct(\ab)$, completing the proof.
\end{proof}
\begin{cor}\label{c:lcttoric}
  For any $u \in P \cap M_\Q$, we have 
  \begin{equation*}
    \lct( D_u) 
    = \inf_{v\in N_\R \setminus \{0 \}} \frac{A(v)}{v(D_u)} 
    = \min_{i=1, \ldots, d} \frac{1}{\langle u, v_i \rangle +b_i}. 
  \end{equation*}
\end{cor}
\begin{proof}
  The first equality follows from Proposition~\ref{p:Tvalcompute}, applied to the 
  the toric graded sequence of ideals defined by $D_u$.
  The functions $v\to A(v)$ and $v\to v(D_u)$ on $N_\R$ are
  both linear on the cones of $\D$, so the function $v\to A(v)/v(D_u)$
  on $N_\R$ attains its infimum at some $v_i$, $1\le i\le d$.
  Since $A(v_i)=1$ and $v_i(D_i)=\langle u, v_i\rangle-\p(v_i)=\langle u, v_i\rangle+b_i$, 
  we are done.
\end{proof}
%
%
%
%
\subsection{Filtrations by toric valuations}
Given $v\in N_\R$, we will describe the filtration $\cF_v$ of $R(X,L)$ and compute 
both $S(v)$ and $T(v)$. 
Recall that for each $m \in\N^*$,
\[
  H^0(X, mL ) =  \bigoplus_{u\in mP \cap M}\C\cdot \chi^u,
\]
where the rational function $\chi^u$ is viewed as a section of $\cO_X(mD)$. 
\begin{prop}\label{p:toricfilt}
  For $\la \in \R_+$ and $m\in \N^*$ we have 
  \[
    \cF_v^\la H^0(X, mL)
    = \bigoplus_{\substack{u\in mP\cap M \\ \langle u, v \rangle - m \cdot \p(v) \ge \la }}
    \C\cdot \chi^u.
  \]
  As a consequence, the set of jumping numbers of $\cF_v$ along $H^0(X,mL)$ 
  is equal to the set $\{ \langle u, v \rangle - m \cdot \p(v)\mid u \in mP \cap M \}$. 
\end{prop}

\begin{proof}
  It suffices to prove that $s=\sum_{u \in mP \cap M} c_u \chi^u\in H^0(X,mL)$,
  then 
  \[
    v(s)
    =\min \{ \langle u, v \rangle - m\cdot \p(v)\mid c_u \neq 0 \}. 
  \]
  To this end, pick $\sigma\in \Delta$ such that $v \in \Int(\sigma)$.
  Note that $\chi^{-m b(\sigma)}$ is a local generator for $\cO_X(mD)$ on $U_\sigma$.  
  By the definition of $v(s)$, and by~\eqref{e126}, we therefore have
  \begin{equation*}
    v(s)
    =v(\sum c_u \chi^{u+mb(\sigma)})
    =\min \{ \langle u, v \rangle + m\langle b(\sigma),v\rangle \mid c_u \neq 0 \},
  \end{equation*}
  which completes the proof, since $\p(v)=-\langle b(\sigma),v\rangle$.
\end{proof} 

\begin{prop}\label{p:SmTm}
  For $m \in \N^*$, we have 
  \[
    S_m(v) = \langle \overline{u}_m,v\rangle - \p(v) 
    \quad\text{and}\quad
    T_m(v) = \max_{u \in P \cap m^{-1} M}\langle u, v \rangle- \p(v),
  \]
  where $\overline{u}_m:= (\sum_{u \in P \cap m^{-1} M} u) / \#(P\cap m^{-1} M)$ 
  is the barycenter of the set $P \cap {m^{-1} M}$.
\end{prop}
\begin{proof}
  From the description of the jumping numbers of $\cF_{v_u}$ in Proposition~\ref{p:toricfilt}, 
  we see 
  \[
    S_m(v)=\frac{\sum_{u \in mP\cap M}\langle u,v\rangle-m\cdot\p(v)}{m\# (mP\cap M)}
    =\left\langle\frac{\sum_{u\in mP\cap M}u}{m\#(mP\cap M)}, v \right\rangle - \p(v),
  \]
  and
  \[
    T_m(v)=\frac{\max_{u\in mP\cap M}\langle u,v \rangle}{m}-\p(v).
  \] 
  Now, multiplication by $m^{-1}$ gives an isomorphism $mP \cap M \to P \cap m^{-1}M$. 
  Applying said isomorphism yields the desired equalities.  
\end{proof}
\begin{cor}\label{c:toricST} 
  We have 
  \begin{equation*}
    S(v) 
    =\langle \overline{u},v\rangle-\p(v) 
    \quad\text{and}\quad
    T(v)
    =\max_{u\in P}\langle u,v\rangle-\p(v) 
    =\max_{u\in\ver(P)}\langle u,v\rangle-\p(v),
  \end{equation*}
  where $\overline{u}$ denotes the barycenter of $P$ and $\ver(P)$ 
  denotes the set of vertices of $P$.  
\end{cor}
\begin{rmk}
  One can thus think of 
  $T(v)=\max_{u\in P}\langle u,v\rangle-\min_{u\in P}\langle u,v\rangle$
  as the width of $P$ in the direction $v$, see also~\cite[\S3.2]{Amb16}.
\end{rmk}
\begin{proof}[Proof of Corollary~\ref{c:toricST}]
  The formula for $S(v)$ is immediate from Proposition~\ref{p:SmTm} since
  $S(v)=\lim_{m\to\infty} S_m(v)$ and $\overline{u}= \lim_{m\to\infty}\overline{u}_m$.
  Similarly, $T(v)=\lim_{m\to\infty}T_m(v)$, and 
  \begin{equation*}
    \lim_{m\to\infty}\max_{u \in P \cap m^{-1} M}\langle u, v \rangle
    =\max_{u \in P}\langle u, v \rangle
    =\max_{u \in\ver P}\langle u, v \rangle,
  \end{equation*}
  where the last equality holds by linearity of $u\mapsto \langle u,v\rangle$.
  This completes the proof.
\end{proof}
\begin{rmk}
  The proof shows that $T_m(v)=T(v)$ for $m$ sufficiently divisible.
\end{rmk}
%
%
%
%
\subsection{Deformation to the initial filtration}
Given a filtration $\cF$ of $R(X,L)$, we will construct a degeneration of $\cF$ to a filtration whose base ideals are $T$-invariant. We will use this construction to show  $\a(L)$ and $\d(L)$ may be computed using only toric valuations. Our argument is a global analogue of~\cite[\S7]{Blu16b}, which in turns draws on~\cite{Musgraded}. 

First write $R(X,L)$ as the coordinate ring of an affine toric variety.
Set $M' := M \times \Z$, $N' := \Hom(M', \Z)$,  $M'_\R :=  M \otimes_\Z \R$, 
and $N'_\R := N \otimes_\Z \R$. 
Let  $\sigma_0$ denote the cone over $P \times \{1 \} \subset M_\R \times \R$. 
Then there is a canonical isomorphism $\C[\sigma_0 \cap M']\simeq R(X,L)$.

We put a $\Z_+^{n+1}$ order on the monomials of $k[\sigma_0 \cap M']$ using 
an argument in [KK14,~\S7].  Choose $y_1, \ldots, y_{n+1} \in \sigma_0^\vee \cap N'$  
that are linearly independent in  $N'_\R$.  Let $\rho\colon M' \to \Z^{n+1}$ denote the map defined by  
\[
  \rho(u)=  \left(  \langle u , y_1 \rangle, \ldots, \langle u, y_{n+1} \rangle \right).
\]
Then $\rho$ is injective and has image contained in $\Z_+^{n+1}$. 

Endowing $\Z_+^{n+1}$ with the lexicographic order gives an order $>$ on the monomials in 
$\C[\sigma_0 \cap M']$. Given an element $s\in\C[\sigma_0 \cap M']$ the \emph{initial term} of $s$, written $\init_{>}(s)$, is the greatest monomial in $s$ with respect to the order $>$. 
Given a subspace  $W$ of $H^0(X,mL)$, we set 
\[
  \init_{>}(W)
  =
  \Span \{ \init_{>}(s)\, \vert \, s  \in W \},
\]
where $W$ is viewed as a vector subspace of $\C[ \sigma_0\cap M']$.
Clearly, $\init_{>}(W)$ is generated by monomials in $\C[\sigma_0 \cap M']$. Therefore, $\bs{ \init_{>}(W)}$ is a $T$-invariant ideal on $X$. 
\begin{prop}\label{p:diminitial}
  If $W$ is a subspace of $H^0(X,mL)$, then $\dim W = \dim \init_{>}(W)$. 
\end{prop}
\begin{proof}
  By construction, there exists a basis of $\init_{>}(W)$ consisting of monomials
  $\chi^{u_1}, \ldots, \chi^{u_r}$, where $u_i\in \sigma_0\cap M'$, 
  and we may assume $\chi^{u_1}>\dots>\chi^{u_r}$.
  For each $1 \le i \le r$, fix $s_i \in W$ such that $\init_{>}(s_i) = \chi^{u_i}$. 
  We claim that $s_1, \ldots, s_r$ forms a basis for $W$. 

  To show that $s_1, \ldots, s_r$ are linearly independent, we argue by contradiction,
  so suppose $0 = \sum_{i=1}^r c_i s_i$, with $c\in\C^r\setminus\{0\}$, and pick $i_0$
  minimal with $c_{i_0}\ne0$. 
  Then $0=\init_{>0}(\sum c_i s_i)=c_{i_0}\chi^{u_{i_0}}$, a contradiction.

  Similarly, if $s_1,\ldots, s_r$ did not span $W$, then there would exist an 
  element $s\in W\setminus \Span \{ s_1,\ldots, s_r\}$ with minimal initial term. 
  Note that $\init_{>}(s) = c \chi^{u_i}$ for some $c \in\C^*$ and $i\in \{1 , \ldots, r\}$. 
  Now,  {$s - c s_i\in W \setminus \Span \{ s_1, \ldots,s_r \}$}, but has 
  initial term strictly smaller than $\init(s)$. This contradicts the minimality 
  assumption on $\init_{>}(s)$, and the proof is complete.
\end{proof}

To understand $\lct( \bs{ \init_> W})$, we construct a 1-parameter degeneration of 
$W$ to $\init_{>}(W)$ essentially following~\cite[\S15.8]{Eisenbud}. 
Choose elements $s_1,\ldots, s_r \in W$ such that 
\[
  W = \Span\{s_1,\ldots, s_r\} 
  \quad\text{and}\quad 
  \init_{>}(W) = \Span \{ \init_{>}(s_1), \ldots, \init_{>}(s_r)\}.
\]
Next, we may fix an integral weight $\mu\colon\sigma_0\cap M \to \Z_+$ such that 
$\init_{>_\mu}(s_i) =\init_{>}(s_i)$
for $1\le i \le r$~\cite[Exercise 15.12]{Eisenbud}. Here $>_\mu$ denotes the 
weight order on $\Z^{n+1}$ induced by $\mu$.  

We write $\C[\sigma_0 \cap M'][t]$ for the polynomial ring in one variable over $\C[\sigma_0\cap M']$. For $s = \sum \beta_u \chi^u \in\C[\sigma_0 \cap M']$, 
we write $d = \max \{ \mu(u) \, \vert \, \beta_m \neq 0 \}$ and set 
\[
  \tilde{s} := t^d \sum \beta_u t^{-\mu(u)} \chi^u.
\]
Next, let $\tilde{W}\subset\C[\sigma_0 \cap M'][t]$ denote the $\C[t]$-submodule of $\C[\sigma_0 \cap M'][t]$ 
generated by $\tilde{s}_1, \ldots, \tilde{s}_r$. 
Then $\tilde{W}$ gives a family of subspaces of $H^0(X,mL)$ over $\A^1$. 
For $c\in \A^1(\C)$, write $W_c$ for the corresponding subspace of $H^0(X,mL)$. 
Clearly $W_1 = W$ and $W_0 = \init_{>}(W)$.
\begin{lem}\label{l:deform}
  For $c\in\C^\ast$, $\lct(\bs{W_c}) = \lct( \bs{W})$. 
\end{lem}
\begin{proof}
  Consider the automorphism of $R(X,L)[t^{\pm 1}]$ defined by 
  $\chi^u \mapsto t^{\mu(u)}\chi^u$ and $t\mapsto t$. 
  Since $X\simeq \Proj(R(X,L))$, this automorphism of $R(X,L)[t^{\pm 1}]$  
  gives an automorphism $X\times (\A^{1}\setminus \{0 \})$ over $\A^1 \setminus \{0\}$. 
  For $c \in\C^*$, we write $\phi_c$ for the corresponding automorphism of $X$. 
  Since $\phi_c^*$ sends $W_c$ to $W$, we see $\lct( \bs{W_c})= \lct(\bs{W})$. 
\end{proof}
\begin{prop}\label{p:lctinitial}
  If $W$ is a subspace of  $H^0(X,mL)$, then $
  \lct( \bs{\init_{>}(W) }) \le \lct(\bs{W})$.
\end{prop}
\begin{proof}
  Combining Proposition~\ref{p:lsclct} with Lemma~\ref{l:deform}, we see 
  $\lct( \bs{ W_0} ) \le \lct (\bs{W})$. Since $\init_>(W) = W_0$, the proof is complete. 
\end{proof}
Let $\cF$ be a filtration of $R(X,L)$. We write $\cF_{\init}$ for the filtration defined by 
\[
\cF_{\init}^\la H^0(X, mL) := \init_{>}\left( \cF^\la H^0(X,mL) \right)
\]
for all $\la \in \R_+$ and $m \in \N$.
To see that $\cF_{\init}$ is indeed a filtration, first note that conditions (F1)--(F3) 
of~\S\ref{ss:filtrations} are trivially satisfied. Condition~(F4) 
follows from the equality $\init_{>}(s_1 s_2) = \init_{>}(s_1) \init_{>}(s_2)$ for $s_1,s_2\in R(X,L)$.
\begin{prop}\label{p:filtinitial}
  With the above setup, we have
  \[
    S(\cF_{in}) = S(\cF), 
    \quad
    T(\cF_{in}) = T(\cF), 
    \quad\text{and}\quad
    \lct(\fb_\bullet(\cF_{\init})) \le \lct(\fb_\bullet(\cF)).\]
\end{prop}
\begin{proof}
  By Proposition~\ref{p:diminitial}, $\cF$ and $\cF_{\init}$ have identical jumping numbers. 
  Thus, $S(\cF)=S(\cF_{\init})$ and $T(\cF) =T(\cF_{\init})$. By Proposition~\ref{p:lctinitial},
  $\lct(\fb_{p,m}(\cF_{\init})) \le \lct(\fb_{p,m})(\cF)$ for $p\in\N$ and $m\in\N$. 
  Letting $m\to\infty$, we get $\lct(\fb_p(\cF_{\init}))\le \lct_p(\fb_\bullet(\cF))$ for all $p\in\N$,
  and hence $\lct(\fb_\bullet(\cF_{\init}))\le \lct(\fb_\bullet(\cF))$. 
\end{proof}
\begin{prop}\label{p:toricval}
  If $w$ is a nontrivial valuation on $X$ with $A(w)<\infty$, 
  then there exists $v \in N_{\R}\setminus\{0\}$ such that 
  \[
    A(v) \le A(w), 
    \quad
    T(v)\ge T(w) , 
    \quad\text{and}\quad 
    S(v) \ge S(w).
  \]
\end{prop}
\begin{proof}
  Let $\cF_{w,\init}$ denote the initial filtration of $\cF_w$. 
  Then $\fb_\bullet(\cF_{w,\init})$ is a graded sequence of $T$-invariant ideals on $X$.
  Further, Proposition~\ref{p:filtinitial} shows that 
  \begin{equation*}
    \lct(\fb_\bullet(\cF_{w,\init})) 
    \le\lct(\fb_\bullet(\cF_w))
    =\lct(\fa_\bullet(w))
    \le A(w)
    <\infty,
  \end{equation*}
  where the first equality Lemma~\ref{l:basevalideals}, and the second inequality is 
  Lemma~\ref{l:lctval}. 

  Therefore, $\fb_\bullet(\cF_{w,\init})$ is a nontrivial graded sequence.
  Proposition~\ref{p:Tvalcompute} yields a nontrivial toric valuation 
  $v \in N_\R$ that computes $\lct(\fb_\bullet(\cF_{w,\init}))$. After rescaling $v$, 
  we may assume $v(\fb_\bullet(\cF_{w,\init}))=1$, and, thus, $A(v)=\lct(\fb_\bullet(\cF_{w,\init}))$. 
  We then have 
  \[
    A(v)
    =\lct(\fb_\bullet(\cF_{w,\init}))
    \le \lct(\fb_\bullet(\cF_w))
    =\lct(\ab(w)) 
    \le A(w),
  \]
  Next, 
  \[
    S(v) \ge S(\cF_{w,\init})=S(\cF_w)=S(w),
  \]
  where the inequality is Corollary~\ref{c:vFinequality} 
  and the following equality is Proposition~\ref{p:filtinitial}. 
  A similar argument gives 
  $T(v)\ge T(w)$ and completes the proof.
\end{proof}
\begin{cor}\label{c:toricthresh}
We have the following equalities
\[
  \a(L) = \inf_{v\in N_\R \setminus \{0 \}} \frac{A(v)}{T(v)}
  \quad\text{and}\quad
  \d(L) = \inf_{v \in N_\R \setminus \{0 \}} \frac{A(v)}{S(v)}
\]
\end{cor}
\begin{proof}
  This is clear from Theorem~C and Proposition~\ref{p:toricval}.
\end{proof}
%
%
%
%
\subsection{Proof of Theorem~F}
We now consider the log canonical and stability thresholds of~$L$. 
The following result is slightly more precise than Theorem~F in the introduction.
\begin{cor}
  We have 
  \begin{equation}\label{e:alphaT}
    \alpha(L) 
    = \min_{u \in \ver(P)} \lct(D_u) 
    =\min_{u\in \ver(P)} \min_{i=1,\ldots, d}\frac{1}{\langle u, v_i\rangle +b_i}
  \end{equation}
  and 
  \begin{equation}\label{e:alphaS}
    \d(L) 
    =\lct(D_{\overline{u}})
    =\min_{i=1,\ldots, d} \frac{1}{\langle \overline{u}, v_i \rangle +b_i},
  \end{equation}
  where $\overline{u}$ denotes the barycenter of $P$ and $\ver(P)$ the set of 
  vertices of $P$. 
  Furthermore, $\a(L)$ (resp. $\d(L)$) is computed by one of the valuations $v_1, \dots, v_d$. 
\end{cor}
\begin{proof}
  Again, we will only prove the half of the corollary that concerns $\alpha(L)$.
  First, we combine Lemma~\ref{l:toriceval}, Corollary~\ref{c:toricST} and 
  Corollary~\ref{c:toricthresh} to see 
  \[
    \alpha(L) 
    =\inf_{v \in N_\R\setminus\{0\}} \min_{u\in \ver(P)}  \frac{A(v)}{v(D_u)} 
    =\min_{u\in\ver(P)}  \inf_{v\in N_\R \setminus \{0 \}} \frac{A(v)}{v(D_u)}. 
  \]
  Applying Corollary~\ref{c:lcttoric} to the previous expression yields~\eqref{e:alphaT}.
  
  Next, pick $u\in\ver(P)$ and $i\in\{1,\ldots, d\}$ such that 
  $\alpha(L)= 1/( \langle u, v_i\rangle+b_i)$. 
  Then we have $A(v_i)/T(v_i)=1/(\langle u, v_i\rangle+b_i)$, so $v_i$ computes $\alpha(L)$.
\end{proof}
%
%
%
%
\subsection{The Fano case}
Finally we consider the case when $X$ is a toric $\Q$-Fano variety, that is, $-K_X$
is an ample $\Q$-Cartier divisor.
\begin{cor}\label{c:kpolybary}
  A toric $\Q$-Fano variety is K-semistable iff the barycenter of the polytope associated
  to $-K_X$ is equal to the origin. 
\end{cor}
This result was proved by analytic methods in~\cite{BB13,Berm16}, even with 
K-semistable replaced by K-polystable, and follows  from~\cite{WZ04} when $X$ is smooth. It can also be deduced 
from~\cite[Theorem 1.4]{LX16}, which is proven algebraically.
\begin{proof}
  We apply~\eqref{e:alphaS} with $b_i=1$ for all $i$. If $\overline{u}=0$, then 
  $\d(-K_X)=1$, which by Theorem~B implies that $X$ is K-semistable. 
  Now suppose $\overline{u}\ne0$. 
  Then $\langle \overline{u},v_i\rangle<0$ for some $i$, or else all the $v_i$
  would lie in a half-space, which is impossible since $\D$ is complete.
  It then follows from~\eqref{e:alphaS} that $\d(-K_X)<1$, so by Theorem~B,
  $X$ is not K-semistable.
\end{proof}
\begin{rmk}
  The proof shows that if $X$ is K-semistable, any toric valuation
  computes $\d(-K_X)=1$.
\end{rmk}

We now give a simple formula for $\d(-K_X)$ in the $\Q$-Fano case. 
When $X$ is smooth, the formula for agrees with the formula in~\cite{Li11} for
the greatest lower bound on the Ricci curvature of $X$, 
as defined and studied in~\cite{Tian92,Sze11}.
\begin{cor}
  Let $X$ be a toric $\Q$-Fano variety and $\overline{u}$ denote the barycenter of
  the polytope
  $P_{-K_X}:=\{u \in M_\R \mid \langle u, v_i \rangle \geq -1\ \text{for all $1\le i \le d$} \}$. 
\begin{enumerate}
\item[(i)] If $X$ is $K$-semistable, then 
$\d(-K_X)=1$. 
\item[(ii)] If $X$ is \emph{not} K-semistable, then 
\[
\delta(-K_X) = \frac{c}{1+c}\]
where $c>0$ is the greatest real number  such that $-c\overline{u}$ lies in $P_{-K_X}$. 
\end{enumerate}
\end{cor}
\begin{proof}
Statement~(i) follows from~\eqref{e:alphaS} and Corollary~\ref{c:kpolybary}. 
For~(ii), we claim that
 \[ 
 0 < \langle \overline{u},v_i \rangle +1 \leq  1/c +1
 \]
 for all $i=1,\ldots,d$ and equality holds in the last inequality for some $i$.
 Statement~(ii) follows from the claim and~\eqref{e:alphaS}. 
 
 We now prove the claim.
 Since $\overline{u}$ lies in the interior of $P_{-K_X}$,  $\langle \overline{u},v_i \rangle > -1$ for all $i$. 
 Since $-c\overline{u}$ lies on the boundary of $P_{-K_X}$, 
 \[
 -c \langle \overline{u},v_i \rangle = \langle -c\overline{u}, v_i \rangle \geq -1
 \]
 for all $i$ and equality holds in the last inequality for some $i$.
 This completes the proof. 
\end{proof}
%
%
%
%
%
%

\end{document}